\documentclass[a4paper]{article}
\usepackage[all]{xy}
\usepackage{graphics,graphicx}
\usepackage{amsfonts,amssymb,amsmath,color,mathrsfs,amstext,bm}
\usepackage{amsbsy, amsopn, amscd, amsxtra, amsthm,authblk}
\usepackage{algorithmic,algorithm,adjustbox}
\usepackage{upref}
\usepackage{geometry}
\usepackage[mathlines]{lineno}
\usepackage{float}
\usepackage{yhmath}
\usepackage{booktabs}
\usepackage{subcaption}
\usepackage{multirow}
\usepackage{makecell}
\usepackage[colorlinks,
            linkcolor=red,
            anchorcolor=red,
            citecolor=red
            ]{hyperref}
\usepackage[nameinlink,noabbrev]{cleveref}
\usepackage{xcolor}
\usepackage[most]{tcolorbox}

\usepackage[shortlabels]{enumitem}

\numberwithin{equation}{section}
\numberwithin{figure}{section}
\numberwithin{table}{section}

\newcommand{\R}{\mathbb{R}}

\newcommand{\dd}{\mathrm{d}}
\newcommand{\cD}{\mathcal{D}}
\newcommand{\pc}{\mathbin{\bar{*}}}

\definecolor{dkgreen}{rgb}{0,0.6,0}
\definecolor{codegray}{rgb}{0.5,0.5,0.5}
\definecolor{mauve}{rgb}{0.58,0,0.82}

\newtcolorbox{mybox}{
  enhanced jigsaw,
  breakable,
  colback=black!0,      % White background.
  colframe=black!100,   % Black border.
  boxrule=0.5pt,
  arc=0pt,
  left=7pt,
  right=7pt,
  top=5pt,
  bottom=5pt
}

\newtheorem{theorem}{Theorem}[section]
\newtheorem{lemma}{Lemma}[section]
\newtheorem{proposition}{Proposition}[section]
\newtheorem{remark}{Remark}[section]

\newtheorem{definition}{Definition}[section]
\newtheorem{assumption}{Assumption}[section]

\makeatletter
\@mparswitchfalse%Put all sidenotes on the right.
\makeatother
\title{A discrete fractional Gr\"onwall inequality for dissipative systems
via resolvent kernel estimate}

\author[a]{
Lei Li\thanks{Email: leili2010@sjtu.edu.cn}
}
\author[a]{Yunxiao Liu \thanks{E-mail: lyx112358@sjtu.edu.cn}}
\author[b]{
Jiwei Zhang  \thanks{Email: jiweizhang@whu.edu.cn}
}

\affil[a]{School of Mathematical Sciences, Institute of Natural Sciences, MOE-LSC, Shanghai Jiao Tong University, Shanghai 200240, China}

\affil[b]{School of Mathematics and Statistics, and Hubei Key Laboratory of Computational Science, Wuhan University, Wuhan, 430072, China}

\date{}

\date{}

\begin{document}

\maketitle

\begin{abstract}
We establish a discrete fractional Gr\"onwall inequality for dissipative systems with time-dependent forcing.
For completely positive discretizations, we bound individual resolvent entries by integrals of the continuous response kernel, under controlled perturbation to the L1 weights and a stepsize restriction.
The proof constructs a supersolution for each resolvent column.
The resulting Mittag--Leffler weights retain both the decay of initial data and the historical contributions of forcing.
The framework covers L1 and several standard schemes on uniform meshes.
For linear subdiffusion, it yields decay-preserving L1 error estimates separating first order and $(2-\alpha)$ order contributions, with stronger long-time decay under additional solution regularity.
\end{abstract}

\section{Introduction}

Time-fractional evolution equations describe dynamics with memory,
in which the present state depends on the entire past history
\cite{diethelm2010analysis,jin2021fractional}.
For dissipative problems, this memory produces algebraic relaxation
and makes the propagation of perturbations different from that in
classical parabolic equations \cite{vergara2015decay}.
A useful stability or error estimate should therefore retain both
the decay of the initial data and the time dependence of external
forcing. This is particularly important in numerical analysis,
where local consistency errors may be large near the initial time
but become smaller as the solution evolves.
The weak initial singularity of solutions and its effect on the
accuracy of the L1 method on graded meshes were analyzed in
\cite{stynes2017error,kopteva2019error}.
Replacing these errors by their maximum over the computational
interval can obscure the effect of dissipation and the resulting
long-time accuracy \cite{zhangzhangzhao-decay}.

The basic continuous model is the scalar fractional relaxation equation
\begin{equation}\label{eq:intro-model}
    D_0^\alpha u(t)+\lambda u(t)=f(t),
    \qquad u(0)=u_0,\qquad 0<\alpha<1,\quad\lambda>0,
\end{equation}
where $D_0^\alpha$ denotes the Caputo derivative with memory starting
at zero (see \cite{diethelm2010analysis,jin2021fractional}, or subsection \ref{subsec:continuouseq} below). Its solution separates into a homogeneous relaxation term
and a convolution of the forcing with a positive response kernel
\cite{diethelm2010analysis,jin2021fractional}:
\begin{equation}\label{eq:intro-representation}
    u(t)=u_0E_\alpha(-\lambda t^\alpha)
       +\lambda^{-1}\int_0^t r_\lambda(t-s)f(s)\,ds,
    \qquad
    r_\lambda(t)=-\frac{d}{dt}E_\alpha(-\lambda t^\alpha).
\end{equation}
Here $E_\alpha$ is the Mittag--Leffler function.
The resolvent kernel $r_\lambda$ determines how a perturbation
introduced at an earlier time contributes to the current solution.
Resolvent methods for Volterra equations with positive kernels have
a classical foundation in
\cite{miller1968volterra,clement1981asymptotic}.
Their use in fractional dynamics includes the monotonicity analysis
of autonomous fractional ordinary differential equations in
\cite{feng2018continuous}.
A more practically useful model is the fractional inequality of the type
\begin{equation}\label{eq:modelineq}
    D_0^\alpha v(t)\le -\lambda v(t)+f(t),
\end{equation}
which can give estimates to some energy functionals and norms for the solutions.
One can bound $v(\cdot)$ by $u$ if $u(0)\ge v(0)$ and such inequalities for bounding $v$ generally belong to the class of Gr\"onwall inequalities \cite{ramirez2012generalized,fllx18note,paper}. 

Let $0=t_0<t_1<\cdots$ be the grid points and $\tau_j=t_j-t_{j-1}$.
Consider the discrete fractional inequality with discrete Caputo operator $\cD_{\tau}^{\alpha}$
\begin{equation}\label{eq:intro-discrete}
\cD_\tau^\alpha v_n+\lambda v_n\le f_n,\qquad   \cD_\tau^\alpha v_n
       :=\sum_{j=1}^n c_{n-j}^n(v_j-v_{j-1}).
\end{equation}
Establishing explicit bounds for the discrete time fractional inequalities, or the discrete Gr\"onwall inequalities, is of great significance to get a priori bounds, to prove stability and to obtain sharp error estimates of numerical schemes.

Several discrete fractional Gr\"onwall inequalities have been established in \cite{feng2018continuous,liao2018sharp,liao2019discrete}. Discrete fractional Gr\"onwall inequalities and convolution-based
consistency estimates for the nonuniform L1 formula were developed in \cite{liao2018sharp}.
The one in \cite{liao2019discrete}  applies to a class of discretizations on nonuniform meshes and allows time-dependent forcing, whose contributions are controlled by weighted sums involving complementary discrete kernels. However, its growth-type estimates involve a Mittag--Leffler factor that increases in time  and are not directly applicable to estimate the decaying rate of solutions. Often the error control for dissipative systems like the subdiffusion problems rely on the maximum principles \cite{stynes2017error} so it cannot get the decay bounds for the numerical solutions.
For dissipative systems, retaining this decay mechanism is important for understanding the long-time behavior of numerical solutions and the propagation of errors. On uniform meshes, results on dissipativity, and Mittag--Leffler stability of numerical solutions have been established in \cite{wangxiaozou2020long,wangzou2023stability} but the results are not easily extended to useful Gr\"onwall inequalities in general cases.
Using the framework of pseudo-convolution \cite{fengli2023a,fengli2024monotonicity}, comparison principles and the first discrete fractional Gr\"onwall inequalities for dissipative systems were established in \cite{paper}, with explicit estimates that capture the effect of dissipation. These dissipative estimates concern constant forcing and provide decay bounds for numerical solutions and uniform-in-time error control. For time-dependent forcing, they can still be applied by replacing the forcing sequence with its maximum over the preceding time levels. The resulting bounds, however, do not retain information about when individual forcing terms act or how their contributions decay over time. This information is particularly important for propagating consistency errors with pronounced time dependence.

In this work, we study dissipative discrete fractional inequalities with time-dependent forcing, aiming to retain both the decay induced by dissipation and the individual historical contributions of the forcing. By comparing individual entries of the discrete resolvent with integrals of the continuous response kernel over the corresponding time intervals
\begin{equation}\label{eq:intro-resolvent}
    0\le (R_\lambda)_{n-m}^n
       \le A\int_{t_{m-1}}^{t_m}r_\lambda(t_n-s)\,ds,
    \qquad 1\le m\le n\le N.
\end{equation}
we obtain a Gr\"onwall inequality with explicit Mittag--Leffler weights that control the contribution of each forcing term to the current solution, providing a tool for decay-preserving error analysis.
We prove this estimate for completely positive discretizations on quasi-uniform meshes, whose weights satisfy specified two-sided
bounds relative to the L1 weights.  The proof constructs a supersolution for each column of the
discrete resolvent.  The assumptions cover the L1 scheme, and some typical complete monotoniicity-preserving schemes on uniform meshes  \cite{lubich1986discretized,liwang2021cm}.

As an application, we derive a decay-preserving temporal error
estimate for the L1 approximation of a linear subdiffusion equation
with dissipative gap $\lambda=\lambda_1-\kappa>0$.
For general smooth forcing on fixed time intervals, the estimate
separates a first-order contribution, weighted by
$(1+\lambda t_n^\alpha)^{-2}$, from a contribution of order
$2-\alpha$, weighted by $(1+\lambda t_n^\alpha)^{-1}$.
Under an additional all-time derivative decay assumption for forcing
approaching a stationary limit, the second weight improves to
$(1+\lambda t_n^\alpha)^{-3}$.
This gives a two-term estimate motivated by the L1 conjecture in
\cite{zhangzhangzhao-decay}, within our regularity, mesh, and
step-size conditions.
It also clarifies how early and recent consistency errors are
propagated differently.
The result concerns temporal error and does not resolve the L2
part of that conjecture or the coarse-mesh regimes outside the
step-size restriction.

The paper is organized as follows.
Section~\ref{sec:setup} introduces the continuous and discrete
resolvents and the pseudo-convolution framework.
Section~\ref{sec:mainresult} contains the main results of the paper,
which states the assumptions and establishes the resolvent estimate
and the resulting Gr\"onwall inequality.
Section~\ref{sec:proof-supersolution} proves the supersolution property.
Section~\ref{sec:standard-schemes} verifies the assumptions for
standard discretizations, and Section~\ref{sec:linear-subdiffusion}
applies the inequality to linear subdiffusion problems.

\section{Preliminaries and problem setup }\label{sec:setup}

In this section, we present some basic definitions, and introduce the relevant concepts and notations.
In particular, we introduce the concepts of resolvent kernels, the pseudo-convolution and the related discrete resolvent
kernels. We then set up the problem to be investigated in this paper.

\subsection{Continuous equation and continuous resolvents }\label{subsec:continuouseq}

Consider the fractional ODE (FODE) \eqref{eq:intro-model} for $\alpha\in (0, 1)$.
The index $0$ indicates that the memory is counted from $t=0$. The Caputo derivative for absolutely continuous function $u$ is defined by
\begin{gather}\label{eq:captra}
D_0^{\alpha}u=\frac{1}{\Gamma(1-\alpha)}\int_0^t\frac{u'(s)}{(t-s)^{\alpha}}ds,
\end{gather}
where $\Gamma(\cdot)$ is the gamma function. Often the comparison principles may involve functions $u$ that are continuous but not absolutely continuous  so the generalized definitions of Caputo derivatives in \cite{liliu18frac,liliu2018compact} might be considered in these cases. The equation is equivalent to the integral equation (see \cite{diethelm2010analysis} and also \cite{liliu18frac,liliu2018compact} for generalized versions)
\begin{gather}\label{eq:integraleq}
u(t)=u_0+\frac{1}{\Gamma(\alpha)}\int_0^{t}(t-s)^{\alpha-1}(-\lambda u(s)+f(s))\,ds.
\end{gather}
In other words, the FODE is equivalent to a Volterra equation with the Abel integral kernel
\begin{gather}\label{eq:kernelfode}
g_{\alpha}(t):=\frac{1}{\Gamma(\alpha)}t_+^{\alpha-1}.
\end{gather}

We review a useful tool for analyzing the fractional derivatives, i.e., the resolvent kernels. 
In \cite{feng2018continuous}, the resolvent kernels (see \cite{clement1981asymptotic,miller1968volterra}) have been used to establish the monotonicity of the solutions to autonomous fractional ODEs. This has then been generalized to discrete schemes on nonuniform meshes in \cite{fengli2023a}.
\begin{definition}\label{def:resol}[see \cite{clement1981asymptotic,miller1968volterra}]
Let $\lambda>0$. The resolvent kernels $r_{\lambda}$ and $s_{\lambda}$ for $a$ are defined respectively by
\begin{gather}
    r_{\lambda}+\lambda r_{\lambda}*a=\lambda a,\quad   s_{\lambda}+\lambda s_{\lambda}* a=1_{t\ge 0}.
\end{gather}
\end{definition}
It has been shown in \cite[Proposition 2.1]{fengli2023a} that both $r_{\lambda}$ and $s_{\lambda}$ are completely monotone for all $\lambda>0$ if $a$ is completely monotone. We recall that a function $a: (0,\infty)\to \R$ is said to be completely monotone if $(-1)^n a^{(n)}(t)\ge 0$ for $n=0,1,\cdots$. If a completely monotone function is not identically zero, it is then strictly positive by the Bernstein theorem \cite{widder41,ssv12}. For the kernel $a=g_{\alpha}$, $r_{\lambda}$ and $s_{\lambda}$ are thus strictly positive. In fact, it has been mentioned in the proof of \cite[Lemma 3.4]{feng2018continuous} that the resolvent $r_{\lambda}$ for $g_{\alpha}=\frac{1}{\Gamma(\alpha)}t_+^{\alpha-1}$ is
\begin{gather}\label{eq:rlambdaDef}
    r_{\lambda}(t)=-\frac{d}{dt}E_{\alpha}(-\lambda t^{\alpha})= \lambda t^{\alpha-1}
      E_{\alpha,\alpha}(-\lambda t^\alpha)>0.
\end{gather}
Here, we recall the Mittag--Leffler function
\begin{equation}\label{eq:powerseriesofml}
    E_{\alpha,\beta}(z)
    :=\sum_{k=0}^{\infty}\frac{z^k}{\Gamma(\alpha k+\beta)},
    \qquad \alpha>0,\quad z\in\mathbb C.
\end{equation}
We write $E_\alpha:=E_{\alpha,1}$ and it is clear that
$E_\alpha'(z)=\alpha^{-1}E_{\alpha,\alpha}(z)$.

Using the definition of resolvents, it is straightforward to find that the solution in \eqref{eq:intro-model} satisfies
\begin{gather}\label{eq:continuousrepresentation}
    \begin{split}
        u(t) &=u_0s_{\lambda}+\lambda^{-1}r_{\lambda}*f\\
        &=u_0 E_{\alpha}(-\lambda t^{\alpha})+\alpha \int_0^t (t-s)^{\alpha-1}E_{\alpha}'(-\lambda(t-s)^{\alpha})f(s)\,ds.
    \end{split}
\end{gather}
Hence,  $s_{\lambda}$ and $r_{\lambda}$ are the fundamental solutions for equation \eqref{eq:intro-model}.

\begin{lemma}\label{d:lem:ml}
The function $r_\lambda$ is positive and decreasing on
$(0,\infty)$, and
\begin{equation}\label{eq:rlambdaasy}
    r_\lambda(t)
    \asymp_\alpha
    \frac{\lambda t^{\alpha-1}}
         {(1+\lambda t^\alpha)^2},
    \qquad t>0.
\end{equation}
\end{lemma}

Here and below, for nonnegative quantities $X$ and $Y$,
we write $X\asymp_{\mathcal P}Y$ if $cY\le X\le CY$
for some positive constants $c$ and $C$ depending only
on the parameters in $\mathcal P$.

\begin{proof}
By \cite{paper}, $r_\lambda$ is positive
and completely monotone, hence decreasing.
The power series definition and the asymptotic expansion
in \cite{gorenflo2002}, with $\beta=\alpha$, give
\begin{equation}
    \lim_{z\to0^+}E_{\alpha,\alpha}(-z)
    =\frac1{\Gamma(\alpha)},
    \qquad
    \lim_{z\to\infty}z^2E_{\alpha,\alpha}(-z)
    =\frac{\alpha}{\Gamma(1-\alpha)}.
\end{equation}
Since $E_{\alpha,\alpha}(-z)$ is positive and continuous
on $[0,\infty)$, these limits imply that
$(1+z)^2E_{\alpha,\alpha}(-z)$ is bounded above and below
by positive constants depending only on $\alpha$.
Thus,
\begin{equation}
    E_{\alpha,\alpha}(-z)
    \asymp_\alpha\frac1{(1+z)^2},
    \qquad z\ge0.
\end{equation}
Taking $z=\lambda t^\alpha$ and using
\eqref{eq:rlambdaDef}
proves \eqref{eq:rlambdaasy}.
\end{proof}

\subsection{Time discretization and pseudoconvolution}

We now consider the discrete operator in \eqref{eq:intro-discrete}.
Let the computational time interval be $[0, T]$, and $0=t_0<t_1<t_2<\cdots<t_N=T$ be the grid points.  Define
\begin{gather}
    \tau_n:=t_n-t_{n-1}, \quad n\ge 1.
\end{gather}
We allow $T=\infty$ and in this case there are infinitely many grid points $t_n$ so that $N=\infty$. Let $u_n$ be the numerical solution at $t_n$. Define the maximum step
\begin{gather}
\tau:=\sup_{n\ge 1}\tau_n.
\end{gather}
Define
\begin{gather}
    \nabla_{\tau}u_j:=u_j-u_{j-1}.
\end{gather}
The operator in \eqref{eq:intro-discrete} is then written as
\begin{gather}\label{eq:DiscreteScheme}
    \cD_{\tau}^{\alpha}u_n:=\sum_{j=1}^n c_{n-j}^n \nabla_{\tau}u_j, \quad n\ge 1.
\end{gather}
We expect that this is an approximation of \eqref{eq:captra}. 
Since $\nabla_{\tau}u_j\approx u'(t_j)\tau_j$, $c_{n-j}^n$ can be understood as an approximation of the average of $g_{1-\alpha}(t_n-s)$ on $[t_{j-1}, t_j)$.

We arrange the kernel $(a_{n-j}^n)$ into a lower triangular array $A$ of the following form
\begin{gather}\label{eq:arraykernel}
    A=\begin{bmatrix}
        a_0^{1} &  &  &  &  \\
        a_1^{2}& a_0^{2} & & & \\
        \cdots & \vdots & \vdots &  &  \\
        a_{n-1}^{n} &  \cdots & a_1^{n} & a_0^{n} &\\
        \cdots & \vdots & \vdots &   & \vdots\\
    \end{bmatrix}.
\end{gather}
Recall the pseudo-convolution in \cite{fengli2023a} (also used in \cite{liao2019discrete,liao2020positive} for defining the ROC and related DCC kernels).   The pseudo-convolution $\pc$ between two array kernels $A, B$ and $C=A\bar{*}B$ is considered in \cite{fengli2023a}, and is given by \begin{gather}\label{eq:convnonuni}
c_{k}^n=\sum_{j=0}^k a_{k-j}^n b_j^{n+j-k}, \quad \text{or}\quad c_{n-k}^n=\sum_{j=k}^n a_{n-j}^n b_{j-k}^j.
\end{gather}
Clearly, when the kernel is arranged in the form \eqref{eq:arraykernel}, the pseudo-convolution is simply the matrix product, with possibly infinite size. Then, with the property of matrices, it is easily verified that the kernel $I$ with $I_{n-j}^n=\delta_{n,j}$ is the identity kernel, with $I\pc A=A\pc I=A$.
If $A$ has nonzero diagonals, then $A$ has an inverse $B$ with $A\pc B=I$, and this kernel $B$ also satisfies $B\pc A=I$ (see the proof in \cite{fengli2023a}).
The pseudo-convolution is also defined between a kernel and a vector $x=(x_1, x_2, \cdots)$ by
\begin{gather}
    y_{n}=\sum_{j=1}^{n}a_{n-j}^n x_{j}.
\end{gather}
then, it holds that
$A\pc(B\pc x)=(A\pc B)\pc x$.
For more details about the pseudo-convolution, one may refer to \cite{fengli2023a, paper}.

With the pseudo-convolution introduced, the scheme above is then written as
\begin{gather}\label{eq:b_and_c}
    \cD_{\tau}^{\alpha}u_n=B\pc (u_n-u_0),n\ge 1, \quad b_0^n=c_0^n, \quad
    b_{n-j}^n=c_{n-j}^n-c_{n-j-1}^n, j\le n-1.
\end{gather}
Suppose that $A$ is the inverse of $B$.
If $\xi_n:=\cD_{\tau}^{\alpha}u_n$, one then has
\begin{gather}
    u_n=u_0+(A\pc \xi)_n=u_0+\sum_{j=1}^n a_{n-j}^n \xi_j.
\end{gather}
This is the discrete analogue of \eqref{eq:integraleq}.

One may consider the resolvent $R_{\lambda}$ defined using the pseudo-convolution similar to the continuous case:
\begin{gather}\label{eq:Rdef}
    R_{\lambda}+\lambda R_{\lambda}\pc A=\lambda A
    \Longleftrightarrow
    A-R_{\lambda}\pc A=\frac{1}{\lambda}R_{\lambda}.
\end{gather}

\begin{definition}
    The kernel $A$ is called completely positive if for any
    $\lambda>0$, $R_{\lambda}$ has nonnegative entries, $0<(R_{\lambda})_0^n<1$
    and $\sum_{j=1}^n (R_{\lambda})_{n-j}^n\le 1$. If $A$ is completely positive, we say the discretization is completely positive.
\end{definition}

The kernel $A$ (or the discretization) is completely positive if and only if $B$ satisfies the following conditions.
\begin{gather}\label{eq:Bcondition}
    \begin{split}
        & b_0^n>0, \quad b_{n-j}^n \le 0, \quad \forall j<n, n\ge 1,\\
        & \sum_{j=1}^n b_{n-j}^n \ge 0, \quad \forall n\ge 1.
    \end{split}
\end{gather} 
Equivalently, in matrix notation, we have
\begin{equation}
    B_{n,n}>0,\quad
    B_{n,j}\le 0\quad (1\le j<n),\quad
    B\mathbf{1}\ge 0,
\end{equation}
where $B_{n,j}=b_{n-j}^n$ and $\mathbf{1}$ denotes the vector
whose entries are all equal to one.

Completely positive discretization enjoys the following comparison principles \cite{paper}.
\begin{theorem}\label{thm:implicitcp}
Suppose that the discretization is completely positive. If $y_0\ge z_0$ and
\begin{equation}
    \cD_{\tau}^{\alpha} y_n\ge F(t, y_n),
    \quad \cD_{\tau}^{\alpha} z_n\le F(t, z_n), \forall n\ge 1,
\end{equation}
then the following comparison principle holds. 
\begin{enumerate}[(1)]
    \item If $F(t,\cdot)$ is nonincreasing for all $t$, then $y_n \ge z_n$ for all $n$.  If $C$ has positive entries, then $y_0>z_0$ implies that $y_n>z_n$.
    \item Suppose that $F(t,\cdot)$ is uniformly Lipschitz such that $|\partial_u F(t, u)|\le M$ for all $t,u$, 
and the discretization satisfies the stability condition $M/c_0^n<1$, then one has $y_n\ge z_n$.  If $c_{n-j}^n>0$, then $y_0>z_0$ implies that $y_n>z_n$.
\end{enumerate}
\end{theorem}

\subsection{The discrete problem and its resolvent representation}

Now, we consider the discrete inequality in \eqref{eq:intro-discrete}. 
We consider a comparison sequence $u_n$ satisfying
\begin{gather}
\cD_{\tau}^{\alpha}u_n=-\lambda u_n+f_n, \qquad u_0=v_0.
\end{gather}
Using the resolvent kernel, one finds that
\begin{gather}\label{eq:discreterepresentation}
    \begin{split}
        u_n&=u_0\pc(I-R_{\lambda})+\lambda^{-1}R_{\lambda}\pc f\\
        &=u_0(1-\sum_{j=1}^n (R_{\lambda})_{n-j}^n)
        +\lambda^{-1}\sum_{m=1}^n (R_{\lambda})_{n-m}^n f_m
    \end{split}
\end{gather}
Clearly, \eqref{eq:discreterepresentation} is the discrete analogue of \eqref{eq:continuousrepresentation} so that $R_{\lambda}$ is in fact the fundamental solution to the discrete numerical scheme.
By the comparison principle in Theorem \ref{thm:implicitcp}, if the discretization is completely positive, one then has
\begin{equation}\label{eq:unvncomp}
    v_n \le u_n=u_0\left(1-\sum_{j=1}^n (R_{\lambda})_{n-j}^n\right) +\lambda^{-1}\sum_{m=1}^n (R_{\lambda})_{n-m}^n f_m.
\end{equation}

Our approach is then to establish some explicit bounds for $R_{\lambda}$ so that the Gr\"onwall's inequality can be established.
We present the intuitive understanding using a continuous analogue here.  We consider the following illustrating problem.
\begin{gather}
    D_0^{\alpha}u=-\lambda u+f
\end{gather}
with $f(t)$ be the piecewise constant interpolation
\begin{gather}
    f(t)=\sum_{i=1}^{\infty} f_i 1_{(t_{i-1}, t_i]}
\end{gather}

Note the following.
\begin{gather}
    \alpha (t-s)^{\alpha-1}E_{\alpha}'(-\lambda (t-s)^{\alpha})
    =-\frac{1}{\lambda}\frac{d}{dt} E_{\alpha}(-\lambda (t-s)^{\alpha})
    =\frac{1}{\lambda}\frac{d}{ds}E_{\alpha}(-\lambda (t-s)^{\alpha}).
\end{gather}
Using \eqref{eq:continuousrepresentation}, the solution $u$ is then given by the following.
\begin{gather}
    u(t)=u_0E_{\alpha}(-\lambda t^{\alpha})+\lambda^{-1}\sum_{m=1}^{\infty}f_m \bar{u}_m(t),
\end{gather}
where
\begin{gather}
    \bar{u}_m(t)=
    \begin{cases}
        1-E_{\alpha}(-\lambda(t-t_{m-1})^{\alpha}), & t\le t_m\\
        E_{\alpha}(-\lambda(t-t_{m})^{\alpha})-E_{\alpha}(-\lambda(t-t_{m-1})^{\alpha}), & t>t_m
    \end{cases}
\end{gather}

It is not hard to see that $\bar{u}_m$ is the solution to the following equation
\begin{equation}
    D_0^{\alpha}\bar{u}_m= -\lambda \bar{u}_m+\lambda 1_{(t_{m-1}, t_m]},\quad \bar{u}_m(t)=0,\quad t\le t_{m-1},
\end{equation}
and admits the representation
\begin{equation}\label{eq:continuous-response-kernel}
    \bar{u}_m(t_n)
    =\int_{t_{m-1}}^{t_m}r_\lambda(t_n-s)\,ds
    =\int_{t_n-t_m}^{t_n-t_{m-1}}r_\lambda(s)\,ds,
    \qquad n\ge m.
\end{equation}
Clearly, $\bar{u}_m$ is absolutely continuous but $(\bar{u}_m)'$ blows up at $t_{m-1}$ and $t_m$.

\section{Resolvent estimates and the Gr\"onwall inequality}\label{sec:mainresult}

In this section, we establish the key estimates of the resolvent kernels and then conclude the discrete 
Gr\"onwall inequality for dissipative systems with varying force terms.
Table~\ref{tab:proof-structure} summarizes the proof strategy by placing the continuous pulse response in \eqref{eq:continuous-response-kernel} alongside the discrete supersolution construction.
The auxiliary kernel $\phi$ has the same algebraic tail as the pulse response $\bar{u}_m$, allowing comparison to transfer the continuous decay to each discrete resolvent entry.

\begin{table}[H]
\centering
\caption{Continuous analogue and discrete proof of the resolvent estimate.}
\label{tab:proof-structure}
\small
\begin{adjustbox}{max width=\textwidth}
\begin{tabular}{@{}lcc@{}}
\toprule
\multicolumn{3}{c}{$\cD_\tau^\alpha v_n+\lambda v_n\le f_n,\qquad\lambda>0$}\\[1mm]
\midrule
Step & Continuous analogue & Discrete proof\\
\midrule
Response &
$\begin{gathered}D_0^\alpha \bar{u}_m+\lambda \bar{u}_m=f_m,\quad \bar{u}_m(0)=0,\\
f_m(t)=\lambda\mathbf 1_{(t_{m-1},t_m]}(t)
\end{gathered}$ &
$(\cD_\tau^\alpha+\lambda)R_{\cdot,m}=\lambda\delta_{\cdot,m}$\\[4mm]
Kernel &
$\begin{gathered}r_\lambda(t)=-\dfrac{d}{dt}E_\alpha(-\lambda t^\alpha),\\[1mm]
r_\lambda(t)\asymp_\alpha\dfrac{\lambda t^{\alpha-1}}{(1+\lambda t^\alpha)^2},\quad t>0
\end{gathered}$ &
$\begin{gathered}\widehat\phi(z)=\dfrac{\lambda}{z^\alpha+\kappa z+\lambda},\\[1mm]
\phi(t)\asymp_\alpha\dfrac{\lambda t^{\alpha-1}}{(1+\lambda t^\alpha)^2},\quad t^{1-\alpha}\ge\kappa
\end{gathered}$\\[4mm]
Construction &
$\begin{gathered}\bar{u}_m(t_n)=\displaystyle\int_{t_{m-1}}^{t_m}r_\lambda(t_n-s)\,ds,\\[1mm]
\bar{u}_m(t)\asymp\dfrac{\tau_m}{\lambda}t^{-1-\alpha}\quad(t\to\infty)
\end{gathered}$ &
$\begin{gathered}W(t)=\begin{cases}(t/\eta)\phi(\eta),&t\le\eta,\\\phi(t),&t>\eta,\end{cases}\\[1mm]
V_m^n=C\tau_m W(t_n-t_{m-1})
\end{gathered}$\\[4mm]
Comparison & \makecell{Continuous reference response\\$\bar{u}_m(t_n)$} &
$\begin{gathered}(\cD_\tau^\alpha+\lambda)V_m\ge\lambda\delta_{\cdot,m},\\[1mm]
0\le R_{n-m}^n\le V_m^n\le A \bar{u}_m(t_n)
\end{gathered}$\\[2mm]
\midrule
\multicolumn{3}{c}{$\displaystyle
v_n\le v_0^+E_\alpha(-\rho^{-1}\lambda t_n^\alpha)
+\frac A\lambda\sum_{m=1}^n f_m^+\bar{u}_m(t_n)$}\\[2mm]
\bottomrule
\end{tabular}
\end{adjustbox}
\end{table}

\subsection{Assumptions and reference gap}
\label{subsec:L1}

Let $0<\alpha<1$ and
$g_{1-\alpha}(t)=t^{-\alpha}/\Gamma(1-\alpha)$ for $t>0$.
Define the L1 weights by
\begin{equation}\label{eq:ellnj}
    \ell_{n,j}
    :=\frac1{\tau_j}\int_{t_{j-1}}^{t_j}
    g_{1-\alpha}(t_n-s)\,ds
    =
    \frac{(t_n-t_{j-1})^{1-\alpha}-(t_n-t_j)^{1-\alpha}}
    {\Gamma(2-\alpha)\tau_j},
    \qquad 1\le j\le n\le N.
\end{equation}
The corresponding L1 discrete operator is
\begin{equation}\label{eq:L1scheme}
    \cD_{\mathrm{L1},\tau}^{\alpha}u_n
    =\sum_{j=1}^n\ell_{n,j}(u_j-u_{j-1}),
\end{equation}
which amounts to taking $c_{n-j}^n=\ell_{n,j}$
in the general discretization.
Since $g_{1-\alpha}$ is decreasing, the reference gap
\begin{equation}
    G_{n,j}
    :=\ell_{n,j}-g_{1-\alpha}(t_n-t_{j-1})
\end{equation}
is nonnegative.
For $j<n$, monotonicity gives
$\ell_{n,j}\le g_{1-\alpha}(t_n-t_j)$, and hence
\begin{equation}
    \frac{G_{n,j}}{\ell_{n,j}}
    \le
    1-\left(\frac{t_n-t_j}{t_n-t_{j-1}}\right)^\alpha
    \le
    \frac{\tau_j}{t_n-t_{j-1}}.
\end{equation}
For $j=n$, the same upper bound follows from
$G_{n,n}\le\ell_{n,n}$.
On the other hand, the convexity of $g_{1-\alpha}$
and Jensen's inequality give
$\ell_{n,j}\ge
g_{1-\alpha}(t_n-t_{j-1}-\tau_j/2)$.
Using the concavity of $x^\alpha$, we obtain
\begin{equation}
    \frac{G_{n,j}}{\ell_{n,j}}
    \ge
    1-\left(1-\frac{\tau_j}{2(t_n-t_{j-1})}\right)^\alpha
    \ge
    \frac{\alpha}{2}\frac{\tau_j}{t_n-t_{j-1}}.
\end{equation}
Combining these bounds yields
\begin{equation}\label{eq:Gnj-comparison}
    \frac{\alpha}{2}
    \frac{\tau_j}{t_n-t_{j-1}}\ell_{n,j}
    \le G_{n,j}
    \le
    \frac{\tau_j}{t_n-t_{j-1}}\ell_{n,j},
    \qquad 1\le j\le n\le N.
\end{equation}

We make the following assumptions for all $1\le j\le n\le N$:

\begin{assumption}\label{ass:proofcond}
There exist constants $K\ge1$, $\nu_{\mathrm{ker}}>0$,
and $M_{\mathrm{ker}}\ge0$ such that the following
conditions hold:
\begin{enumerate}[label=(\roman*)]
    \item\label{cond1} \textbf{Complete positivity.}
    \begin{equation}\label{ass:cp}
        B_{n,n}>0,\qquad
        B_{n,j}\le0\quad (j<n),\qquad
        B\mathbf{1}\ge0.
    \end{equation}

    \item\label{cond2} \textbf{Mesh condition.}
    \begin{equation}\label{ass:mesh}
        \frac{\max_{1\le j\le N}\tau_j}
             {\min_{1\le j\le N}\tau_j}\le K.
    \end{equation}

    \item\label{cond3}
    \textbf{Two-sided kernel bounds.}
    For all $1\le j\le n\le N$,
    \begin{equation}\label{cond4*}
        \max\{\nu_{\mathrm{ker}},1-\varepsilon_{n,j}\}
        \ell_{n,j}
        \le c_{n-j}^n
        \le(1+\varepsilon_{n,j})\ell_{n,j},
    \end{equation}
    where
    \begin{equation}\label{eq:epsilon}
        \varepsilon_{n,j}
        :=M_{\mathrm{ker}}
          \frac{\tau_j}{t_n-t_{j-1}}.
    \end{equation}
\end{enumerate}
The constants $K$, $\nu_{\mathrm{ker}}$, and
$M_{\mathrm{ker}}$ are independent of $\lambda$, $T$,
$N$, and the mesh refinement.
\end{assumption}

\begin{theorem}\label{thm:mainest}
Let $\lambda>0$ be given.
Suppose that Assumption~\ref{ass:proofcond} holds
and that the time steps satisfy
\begin{equation}\label{eq:stepsizeconstraint}
    \lambda\tau^\alpha\le1.
\end{equation}
Then there exists a constant $A>0$, depending only on
$\alpha$, $K$, $\nu_{\mathrm{ker}}$, and $M_{\mathrm{ker}}$,
such that
\begin{equation}\label{eq:resolvent-pointwise-bound}
    0\le R_{n-m}^n
    \le\min\{\beta_0,A \bar{u}_m(t_n)\},
    \qquad 1\le m\le n\le N,
\end{equation}
where $\beta_0=R_0^m$.
\end{theorem}

\subsection{The equation for the discrete resolvent}

Multiplying \eqref{eq:Rdef} on the right by $B$
and using $A\pc B=I$, we obtain, in matrix notation,
\begin{equation}
    R_\lambda(B+\lambda I)=\lambda I.
\end{equation}
Since these are square matrices and $\lambda>0$,
it follows that $R_\lambda=\lambda(B+\lambda I)^{-1}$.
Hence,
\begin{equation}\label{eq:BandR}
    (B+\lambda I)R_\lambda=\lambda I.
\end{equation}

We write $R_{n-m}^n=(R_\lambda)_{n-m}^n$.
For a fixed $m$, taking the $(m,m)$ entry gives
\begin{equation}\label{eq:beta0def}
    \beta_0:=R_0^m
    =\frac{\lambda a_0^m}{1+\lambda a_0^m}
    =\frac{\lambda}{c_0^m+\lambda},
\end{equation}
where $a_0^m=1/b_0^m=1/c_0^m$.
Although denoted by $\beta_0$, this quantity depends on
the fixed index $m$ and the parameter $\lambda$.
Condition~\ref{cond3} gives
\begin{equation}\label{eq:c0mest}
    \frac{\nu_{\mathrm{ker}}\tau_m^{-\alpha}}
         {\Gamma(2-\alpha)}
    \le c_0^m  \le  \frac{(1+M_{\mathrm{ker}})\tau_m^{-\alpha}}
         {\Gamma(2-\alpha)}.
\end{equation}

Taking the $(n,m)$ entry of
\eqref{eq:BandR}, we obtain
\begin{equation}\label{eq:bRidentity}
    \sum_{j=m}^n b_{n-j}^n R_{j-m}^j
    +\lambda R_{n-m}^n
    =\lambda\delta_{n,m},
    \qquad n\ge m.
\end{equation}
This is the equation for the resolvent kernel. Extending the $m$th column by zero for $0\le j<m$
and using \eqref{eq:b_and_c}, we can equivalently write
\begin{equation}\label{eq:Ridentity}
    \sum_{j=m}^n c_{n-j}^n
    \bigl(R_{j-m}^j-R_{j-1-m}^{j-1}\bigr)
    +\lambda R_{n-m}^n
    =\lambda\delta_{n,m},
    \qquad n\ge m,
\end{equation}
where the term $R_{-1}^{m-1}$ is understood to be zero.

By \eqref{cond4*}, the weights satisfy
\begin{equation}
    \nu\ell_{n,j}\le c_{n-j}^n\le\rho\ell_{n,j},
\end{equation}
where we take
$\rho=(1+M_{\mathrm{ker}})/(1-\alpha)$ and
$\nu=\min\{1,\nu_{\mathrm{ker}}\}$,
so that $\rho\ge1\ge\nu>0$.
Applying the results of \cite{paper} with
$v_0=1$ and $f_n=0$, or with
$v_0=0$ and $f_n=\mathbf{1}_{\{n\ge m\}}$,
yields the following cumulative bounds for $R$.
\begin{lemma}\label{eq:sumR}
 With the stated assumptions, one has
    \begin{equation}
        0\le 1-\sum_{j=1}^n R_{n-j}^n \le  E_{\alpha}(-\rho^{-1}\lambda t_n^{\alpha}).
    \end{equation}
 Moreover, one also has
    \begin{equation}
        \lambda^{-1}\sum_{j=m}^n R_{n-j}^n \le \frac{1}{\lambda}[1- E_{\alpha}(-\nu^{-1}\lambda (t_n-t_{m-1})^{\alpha})].
    \end{equation}
\end{lemma}

Combining conditions~\ref{cond1} and~\ref{cond3}
of Assumption~\ref{ass:proofcond}
with \eqref{eq:b_and_c}, we obtain
\begin{equation}\label{eq:cmono}
    0<c_{n-1}^n\le c_{n-2}^n\le\cdots\le c_0^n,
    \qquad n\ge 1.
\end{equation}
\begin{lemma}\label{lem:Runifbound}
Let $\lambda>0$. Suppose that, for all $1\le n\le N$, \eqref{eq:cmono} holds.
If $w_0\ge0$ and
\begin{equation}\label{eq:wseq}
    \cD_\tau^\alpha w_n+\lambda w_n\ge0,
    \qquad n\ge 1,
\end{equation}
then $w_n\ge0$. Moreover, the resolvent kernel satisfies
\begin{equation}\label{eq:RInitBound}
    0\le R_{n-m}^n\le\beta_0,
    \qquad 1\le m\le n,
\end{equation}
where $\beta_0=R_0^m=\lambda/(c_0^m+\lambda)$.
\end{lemma}

\begin{proof}
We first prove the comparison principle.
Expanding \eqref{eq:wseq} using \eqref{eq:DiscreteScheme}, we obtain
\begin{equation}\label{d:eq:expand}
    0
    \le \cD_\tau^\alpha w_n+\lambda w_n
    =(c_0^n+\lambda)w_n
    -\sum_{j=1}^{n-1}
    (c_{n-j-1}^n-c_{n-j}^n)w_j
    -c_{n-1}^n w_0.
\end{equation}
Induction from $w_0\ge0$
gives $w_n\ge0$ for all $1\le n\le N$.

For the resolvent bound, fix $1\le m\le N$.
At $n=m$, we have
$R_0^m=\beta_0>0$.
Let $n>m$ and suppose that
$0\le R_{j-m}^j\le\beta_0$ for all $m\le j<n$.
Expanding \eqref{eq:bRidentity} and using the weight monotonicity,
we obtain
\begin{equation}\label{d:eq:capinduction}
    \begin{aligned}
        0
        &\le(c_0^n+\lambda)R_{n-m}^n
        =\sum_{j=m}^{n-1}
        (c_{n-j-1}^n-c_{n-j}^n)R_{j-m}^j\\
        &\le\beta_0\sum_{j=m}^{n-1}
        (c_{n-j-1}^n-c_{n-j}^n)
        =\beta_0(c_0^n-c_{n-m}^n)
        \le\beta_0(c_0^n+\lambda).
    \end{aligned}
\end{equation}
Hence, $0\le R_{n-m}^n\le\beta_0$. This completes the induction.
\end{proof}

The estimate above are not very useful if one has varying forcing terms. In some applications, $f_n$ may decay with $n$.
To recover the accurate decay estimate (as in section \ref{sec:linear-subdiffusion}), one has to obtain the entrywise bound for $(R_{\lambda})_{n-m}^n$, and ideally by $\bar{u}_m(t_n)$.

\subsection{The supersolution and Gr\"onwall inequality}

In this subsection, we present the construction of the supersolution so that Theorem~\ref{thm:mainest} holds.

Consider the solution to the following fractional equation
\begin{equation}\label{eq:contKernelequation}
    (D_{\mathrm{RL}}^\alpha+\lambda)\phi_{\lambda,\kappa}(t)
    =-\kappa \phi_{\lambda,\kappa}'(t), \quad \phi_{\lambda, \kappa}(0)=\lambda/\kappa.
\end{equation}
where
$D_{\mathrm{RL}}^\alpha \phi_{\lambda,\kappa}
=(I^{1-\alpha}\phi_{\lambda,\kappa})'$
denotes the Riemann--Liouville derivative. 
\begin{remark}
The nonnegative residual $-\kappa\phi_{\lambda,\kappa}'$
in \eqref{eq:contKernelequation}
will be used to control the discretization errors.
\end{remark}

Recall $\tau=\sup_j\tau_j$. Define
\begin{equation}\label{eq:kappa}
    \kappa_m:=D_*\tau^{1-\alpha},
\end{equation}
where $D_*\ge1$ and $L_*\ge4$ will be chosen in Subsection~\ref{subsec:step4} so that
\begin{equation}\label{eq:D*L*constraint}
    D_*L_*^{-(1-\alpha)}\le1.
\end{equation}
We remark that $\kappa_m$ is a small parameter that tends to zero as $\tau\to 0$.
Let $\eta_m$ be the smallest shifted grid point such that 
\begin{gather}\label{eq:etam}
L_*\tau \le \eta_m=t_{j_*}-t_{m-1}<(L_*+1)\tau.
\end{gather}
\begin{remark}
If only finitely many grid points are considered (i.e., $N<\infty$), one may extend the grid by introducing some ghost points with fixed step (for example $\tau_N$) to define $\eta_m$ when $m$ is close to $N$. Note that this will not affect the analysis for the focused grid points.
\end{remark}

Write $\phi_m:=\phi_{\lambda,\kappa_m}$ and define
\begin{equation}\label{eq:WmZmDef}
    \mathcal W_m(t):=
    \begin{cases}
        \dfrac{t}{\eta_m}\phi_m(\eta_m),
        &0\le t\le\eta_m,\\[1mm]
        \phi_m(t),&t>\eta_m.
    \end{cases}
\end{equation}
Introduce the shifted nodes
\begin{equation}
    x_j:=t_j-t_{m-1},
    \qquad j=m-1,\ldots,N.
\end{equation}
Define
\begin{equation}\label{eq:Wmjdef}
    \mathcal W_m^j:=
    \begin{cases}
        0,&j<m,\\
        \mathcal W_m(x_j),&j\ge m.
    \end{cases}
\end{equation}
and
\begin{equation}\label{eq:V}
    \mathcal V_m^n
    :=\frac{2\tau_m}{\gamma_*}\mathcal W_m^n,
\end{equation}
where
\begin{equation}\label{eq:sourceconstant}
    \gamma_*:=
    c_\alpha\bigl((L_*+1)K\bigr)^{-\alpha-2}
    \min\left\{
        \frac{\nu_{\mathrm{ker}}}{\Gamma(2-\alpha)},1
    \right\}>0.
\end{equation}
Here, $c_{\alpha}>0$ is a constant depending on $\alpha$ only such that
\[
\phi_{\lambda, \kappa}(t)\ge c_{\alpha}\frac{\lambda t^{\alpha-1}}{(1+\lambda t^{\alpha})^2},
\quad t^{1-\alpha}\ge \kappa,
\]
as shown in Lemma \ref{lem:kernelcont}.

We will establish the following estimate, whose proof is deferred to Section~\ref{sec:proof-supersolution}.
\begin{lemma}\label{lem:V}
Suppose that Assumption~\ref{ass:proofcond} holds.
For each fixed $1\le m\le N$, let $\lambda>0$ satisfy
$\lambda\tau_m^\alpha\le1$.
Then the sequence
$\{\mathcal V_m^n\}_{n=0}^N$, with
$\mathcal V_m^n=0$ for $0\le n<m$ satisfies
\begin{equation}\label{eq:DVest}
        \cD_\tau^\alpha\mathcal V_m^n
        +\lambda\mathcal V_m^n
        \ge\lambda\delta_{n,m},
        \qquad m\le n\le N,
\end{equation}
\begin{equation}\label{eq:Vbound}
        0\le\mathcal V_m^n
        \le C_{\mathrm{bar}}\bar{u}_m(t_n),
        \qquad m\le n\le N.
\end{equation}
Here, $C_{\mathrm{bar}}$ depends only on
$\alpha$, $K$, $\nu_{\mathrm{ker}}$, and $M_{\mathrm{ker}}$.
\end{lemma}

We now take the lemma to derive the main result.
\begin{proof}[Proof of Theorem~\ref{thm:mainest}]
The comparison principle in Lemma~\ref{lem:Runifbound} then gives
\begin{equation}
    R_{n-m}^n
    \le\mathcal V_m^n
    \le C_{\mathrm{bar}}\bar{u}_m(t_n),
    \qquad m\le n.
\end{equation}

\end{proof}

Consider again the discrete inequality \eqref{eq:intro-discrete}.
The initial value $v_0$ and the forcing terms $f_n$ may be arbitrary real numbers. Write $x^+:=\max\{x,0\}$.

As a consequence of the above result, combining complete positivity and \eqref{eq:unvncomp}
with Theorem~\ref{thm:mainest} and Lemma~\ref{eq:sumR}, one has the following.
\begin{theorem}
Suppose that Assumption~\ref{ass:proofcond} holds
and that the time steps satisfy \eqref{eq:stepsizeconstraint}.
The sequence in \eqref{eq:intro-discrete} then satisfies for $1\le n\le N$ that
\begin{equation}\label{eq:discrete-gronwall}
    \begin{aligned}
        v_n
        &\le
        v_0^+\left(1-\sum_{m=1}^nR_{n-m}^n\right)
        +\frac1\lambda\sum_{m=1}^nR_{n-m}^n f_m^+\\
        &\le
        v_0^+E_\alpha(-\rho^{-1}\lambda t_n^\alpha)
        +
        \frac A\lambda\sum_{m=1}^n f_m^+
        \Bigl(
            E_\alpha\bigl(-\lambda(t_n-t_m)^\alpha\bigr)
            -
            E_\alpha\bigl(-\lambda(t_n-t_{m-1})^\alpha\bigr)
        \Bigr).
    \end{aligned}
\end{equation}
Here, $A$ is the constant in Theorem~\ref{thm:mainest},
and $\rho$ is the constant from
Lemma~\ref{eq:sumR}.
\end{theorem}

\section{Proof of the supersolution property}
\label{sec:proof-supersolution}

In this section, we give the proof of Lemma~\ref{lem:V}.
We first present some properties of $\phi_{\lambda,\kappa}$
and the associated cumulative kernel, defined by
\begin{equation}
    Q_{\lambda,\kappa}(t)
    :=\int_0^t \phi_{\lambda,\kappa}(s)\,\dd s,
\end{equation}
We denote the Laplace transform of $g$ by $\widehat g$.
We have the following result about the behaviors of the functions and the proofs are provided in section \ref{app:phi}.
\begin{lemma}\label{lem:kernelcont}
Let $\lambda,\kappa>0$. The Laplace transform of $\phi_{\lambda,\kappa}$ is given by
\begin{equation}\label{eq:phiLap}
    \widehat \phi_{\lambda,\kappa}(z)
    =\frac{\lambda}{z^\alpha+\kappa z+\lambda},
    \qquad z>0.
\end{equation}
Then $\phi_{\lambda,\kappa}$ is positive, decreasing, and convex
on $(0,\infty)$, and satisfies
\begin{equation}\label{eq:phiconc1}
    \int_0^\infty \phi_{\lambda,\kappa}(t)\,\dd t=1.
\end{equation}
For every $t$ with $t^{1-\alpha}\ge \kappa$,
we have
\begin{equation}\label{eq:phiasymp}
    (-1)^p \phi_{\lambda,\kappa}^{(p)}(t)
    \asymp_\alpha
    \frac{\lambda t^{\alpha-1-p}}
         {(1+\lambda t^\alpha)^2},
    \qquad p=0,1,2.
\end{equation}
The kernel $Q_{\lambda,\kappa}$ satisfies
\begin{equation}\label{eq:Qasym}
    Q_{\lambda,\kappa}(t)
    \asymp_\alpha
    \frac{\lambda t^\alpha}{1+\lambda t^\alpha}
\end{equation}
under the same condition $t^{1-\alpha}\ge \kappa $.
The constants in \eqref{eq:phiasymp} and \eqref{eq:Qasym}
are independent of $\lambda$, $\kappa$, and $t$. 
\end{lemma}

\subsection{Continuous residual}\label{subsec:step2}

To start the verification of the supersolution property, we start with applying the continuous operator on $\mathcal W_m$.
Note that $\mathcal W_m$ is locally absolutely continuous
with $\mathcal W_m(0)=0$, so its Caputo and
Riemann--Liouville derivatives coincide. Using \eqref{eq:contKernelequation}, we obtain for $t>\eta_m$ that
\begin{equation}\label{eq:continuousres}
    \begin{aligned}
        (D_{\mathrm C}^\alpha+\lambda)\mathcal W_m(t)
        &=(D_{\mathrm{RL}}^\alpha+\lambda)\phi_m(t)
          -D_{\mathrm{RL}}^\alpha\mathcal Z_m(t)\\
        &=\mathcal H_m(t)-\kappa_m\phi_m'(t)>0,
        \qquad t>\eta_m.
    \end{aligned}
\end{equation}
Here, 
\begin{equation}\label{eq:HmDef}
    \mathcal H_m(t)
    :=-D_{\mathrm{RL}}^\alpha\mathcal Z_m(t)
    =\frac{\alpha}{\Gamma(1-\alpha)}
      \int_0^{\eta_m}
      \mathcal Z_m(s)(t-s)^{-1-\alpha}\,\dd s
    >0,
\end{equation}
with $Z_m(t):=\phi_m(t)-\mathcal W_m(t)$.
Since $\phi_m$ is decreasing, $\mathcal Z_m\ge0$,
with strict positivity on $(0,\eta_m)$,
and $\mathcal Z_m=0$ on $[\eta_m,\infty)$.

By construction and the mesh condition indicated in \eqref{eq:etam},
\begin{equation}\label{eq:etamIneq}
    1\le\frac{\eta_m}{\tau_m}\le(L_*+1)K.
\end{equation}
Hence, when $t\ge \eta_m$, one has
\begin{equation}
    \kappa_m t^{\alpha-1} =D_*\left(\frac{\tau}{t}\right)^{1-\alpha}  \le D_*L_*^{-(1-\alpha)}   \le1.
\end{equation}
Moreover, since $Q_{\lambda,\kappa_m}(\eta_m)\ge\eta_m\phi_m(\eta_m)$,
\begin{equation}\label{eq:ZmItegral}
    \int_0^{\eta_m}\mathcal Z_m(s)\,\dd s
    =Q_{\lambda,\kappa_m}(\eta_m)
     -\frac12\eta_m\phi_m(\eta_m)
    \ge\frac12Q_{\lambda,\kappa_m}(\eta_m).
\end{equation}
Combining this with \eqref{eq:HmDef},
\eqref{eq:Qasym}, and
$(t-s)^{-1-\alpha}\ge t^{-1-\alpha}$ yields
\begin{equation}\label{eq:Hmlower}
    \mathcal H_m(t)
    \ge c_\alpha
    Q_{\lambda,\kappa_m}(\eta_m)t^{-1-\alpha}
    \ge c_\alpha
    \frac{\lambda\eta_m^\alpha}
         {1+\lambda\eta_m^\alpha}
    t^{-1-\alpha},
    \qquad t>\eta_m.
\end{equation}
We also have
\begin{equation}\label{eq:Hmlower2}
    \mathcal H_m(t)
    \ge c_\alpha\eta_m\phi_m(\eta_m)t^{-1-\alpha},
    \qquad t>\eta_m.
\end{equation}

Next, we will move onto the discrete case. Using the continuous computation here, we aim to bound the discrete
derivative by $\mathcal H_m$ below.
Set the difference of the $c_{n_j}^n$ coefficient between the reference $L1$ scheme and the focused scheme
\[
d_{n,j}:=\ell_{n,j}-c_{n-j}^n.
\]
We then decompose the operator into the following
\begin{equation}\label{eq:Wmnfrac}
    (\mathcal D_\tau^\alpha\mathcal W_m)^n
    =(\mathcal D_{\mathrm{L1},\tau}^\alpha\mathcal W_m)^n
    -\sum_{j=m}^n d_{n,j}
    (\mathcal W_m^j-\mathcal W_m^{j-1}).
\end{equation}
Note that condition~\ref{cond3} of Assumption~\ref{ass:proofcond}
and \eqref{eq:Gnj-comparison} give
\begin{equation}\label{eq:weight-deviation-bound}
    |d_{n,j}|
    \le M_{\mathrm{ker}}
        \frac{\tau_j}{t_n-t_{j-1}}\ell_{n,j}
    \le\frac{2M_{\mathrm{ker}}}{\alpha}G_{n,j}.
\end{equation}
Below, we first focus on the $L1$ discretization in subsection \ref{subsec:L1-error} and then move to the 
scheme in subsection \ref{subsec:step4}.

\subsection{The reference L1 discretization error estimate}
\label{subsec:L1-error}

Let $\mathcal I_m\mathcal W_m$ denote the piecewise
linear interpolant on the shifted grid, and set
\begin{equation}\label{eq:e}
    \mathfrak e_m(s)
    :=(\mathcal I_m\mathcal W_m)(s)-\mathcal W_m(s).
\end{equation}
Since $\eta_m$ is a grid node,
$\mathfrak e_m=0$ on $[0,\eta_m]$ and
$\mathfrak e_m\ge0$ on the tail by convexity.
Moreover, $\mathfrak e_m$ vanishes at every grid node.

Fix $x_n>\eta_m$. The L1 reference operator satisfies
\begin{equation}\label{eq:L1op}
    (\mathcal D_{\mathrm{L1},\tau}^\alpha\mathcal W_m)^n
    :=\sum_{j=m}^n
    \ell_{n,j}(\mathcal W_m^j-\mathcal W_m^{j-1})
    =D_{c}^\alpha
    (\mathcal I_m\mathcal W_m)(x_n),
\end{equation}
where $D_c^{\alpha}$ indicates the Caputo derivative starting from $s=0$ (or $t=t_{m-1}$).

Using $\mathfrak e_m(\eta_m)=\mathfrak e_m(x_n)=0$
and $\mathfrak e_m(s)=O(t-s)$ as $s\to x_n^{-}$,
integration by parts gives
\begin{equation}
    (\mathcal D_{\mathrm{L1},\tau}^\alpha\mathcal W_m)^n
    -D_{c}^\alpha\mathcal W_m(x_n)
    =D_{c}^\alpha\mathfrak e_m(x_n)
    = -E_m(x_n),
\end{equation}
where
\begin{equation}\label{eq:Edef}
   E_m(x_n) :=\frac{\alpha}{\Gamma(1-\alpha)}
    \int_{\eta_m}^{x_n} \mathfrak e_m(s)(x_n-s)^{-1-\alpha}\,\dd s
    \ge0.
\end{equation}
Combining this identity with \eqref{eq:continuousres},
we obtain
\begin{equation}\label{eq:WmnL1exp}
    (\mathcal D_{\mathrm{L1},\tau}^\alpha\mathcal W_m)^n
    +\lambda\mathcal W_m^n
    =\mathcal H_m(x_n)-\kappa_m\phi_m'(x_n)-E_m(x_n).
\end{equation}

We next estimate $E_m$.
Define
\begin{equation}\label{eq:Flambda}
    \mathcal F_\lambda(t)
    :=\frac{\lambda t^{\alpha-3}}
            {(1+\lambda t^\alpha)^2},
    \qquad t>0.
\end{equation}
Since $0<\alpha<1$, the function $\mathcal F_\lambda$
is positive and decreasing.
Let $0<\vartheta\le1$ and $\vartheta t\le u\le t$.
Since
$1+\lambda u^\alpha\ge\vartheta^\alpha(1+\lambda t^\alpha)$,
we obtain
\begin{equation}\label{d:eq:Fscale}
    \mathcal F_\lambda(u)
    =
    \left(\frac ut\right)^{\alpha-3}
    \left(\frac{1+\lambda t^\alpha}
               {1+\lambda u^\alpha}\right)^2
    \mathcal F_\lambda(t)
    \le
    \vartheta^{\alpha-3}\vartheta^{-2\alpha}
    \mathcal F_\lambda(t)
    =
    \vartheta^{-3-\alpha}\mathcal F_\lambda(t).
\end{equation}
Moreover, for any $\eta>0$, monotonicity of the denominator gives
\begin{equation}\label{d:eq:Ftail}
    \int_\eta^\infty\mathcal F_\lambda(u)\,du
    \le
    \frac{\lambda}{(1+\lambda\eta^\alpha)^2}
    \int_\eta^\infty u^{\alpha-3}\,du
    =
    \frac{\lambda\eta^{\alpha-2}}
         {(2-\alpha)(1+\lambda\eta^\alpha)^2}.
\end{equation}
By \eqref{eq:phiasymp},
$\phi_m''(s)\le C_\alpha\mathcal F_\lambda(s)$
for $s\ge\eta_m$.
For a grid cell $[a,b]\subseteq[\eta_m,x_n]$ containing $s$,
the linear interpolation error satisfies
\begin{equation}\label{eq:errLinear}
    0\le\mathfrak e_m(s)
    \le\frac12(s-a)(b-s)
    \sup_{a\le\xi\le b}\phi_m''(\xi)\le \frac{\tau^2}{8} C_\alpha\mathcal F_\lambda(a).
\end{equation}
Since $\eta_m\ge4\tau$, the scaling estimate
\eqref{d:eq:Fscale} yields
\begin{equation}\label{eq:elower1}
    \mathfrak e_m(s)
    \le C_\alpha\tau^2\mathcal F_\lambda(s),
    \qquad \eta_m\le s\le x_n.
\end{equation}
For $s\ge x_n/2$, we also have $a\ge x_n/4$ and $b\le x_n$.
Retaining the factor
$b-s\le\min\{\tau, x_n-s\}$ in \eqref{eq:errLinear}
therefore gives
\begin{equation}\label{eq:elower2}
    \mathfrak e_m(s)
    \le C_\alpha\mathcal F_\lambda(x_n)\,
    \tau\min\{\tau,x_n-s\},
    \qquad \max\{\eta_m,t/2\}\le s\le x_n.
\end{equation}

Split $E_m=E_m^{(1)}+E_m^{(2)}$ over $[\eta_m, x_n/2]$ and
$[\max\{\eta_m, x_n/2\},x_n]$, respectively,
with $E_m^{(1)}=0$ when $x_n/2<\eta_m$.
By \eqref{eq:elower1}, \eqref{d:eq:Ftail},
and \eqref{eq:Hmlower},
\begin{equation}\label{eq:E1Upp}
E_m^{(1)}(x_n)  \le C_\alpha\tau^2x_n^{-1-\alpha}
    \int_{\eta_m}^\infty\mathcal F_\lambda(s)\,\dd s
    \le C_\alpha  \left(\frac{\tau}{\eta_m}\right)^2
    \mathcal H_m(x_n).
\end{equation}
Similarly, using \eqref{eq:elower2} and
\eqref{eq:Hmlower}, with $y=x_n-s$ and $x_n\ge\eta_m$, gives
\begin{equation}\label{eq:E2Upp}
    \begin{aligned}
   E_m^{(2)}(x_n)   &\le C_\alpha\mathcal F_\lambda(x_n)\tau  \int_0^\infty \min\{\tau,y\}y^{-1-\alpha}\,\dd y\\
        &\le C_\alpha\mathcal F_\lambda(x_n)\tau^{2-\alpha}
        \le C_\alpha  \left(\frac{\tau}{\eta_m}\right)^{2-\alpha}  \mathcal H_m(x_n).
    \end{aligned}
\end{equation}
Combining the two estimates yields
\begin{equation}\label{eq:Emupper}
E_m(x_n)    \le C_\alpha  \left(  \left(\frac{\tau}{\eta_m}\right)^2   +\left(\frac{\tau}{\eta_m}\right)^{2-\alpha}
    \right)\mathcal H_m(x_n).
\end{equation}

\subsection{Weight perturbation and error absorption}\label{subsec:step4}

We now return to the general scheme \eqref{eq:Wmnfrac}. Fix again $x_n=t_n-t_{m-1}>\eta_m$.  By \eqref{eq:WmZmDef} and \eqref{eq:Wmjdef},
\begin{equation}
    |\mathcal W_m^j-\mathcal W_m^{j-1}|  = \begin{cases}
        \dfrac{\tau_j}{\eta_m}\phi_m(\eta_m),   &x_j\le\eta_m,\\
        \displaystyle\int_{x_{j-1}}^{x_j}  (-\phi_m'(s))\,\dd s, &x_{j-1}\ge\eta_m.
    \end{cases}
\end{equation}
The triangle inequality and \eqref{eq:weight-deviation-bound} give
\begin{equation}\label{eq:weight-error-bound}
    \left|
    \sum_{j=m}^n d_{n,j}
    (\mathcal W_m^j-\mathcal W_m^{j-1})
    \right|
    \le \mathcal{J}_m^{(1)}(x_n)+\mathcal{J}_m^{(2)}(x_n),
\end{equation}
where the nonnegative error terms $\mathcal{J}_m^{(1)}(x_n)$ and $\mathcal{J}_m^{(2)}(x_n)$
are given respectively as follows.
\begin{align}
    \mathcal{J}_m^{(1)}(x_n)
    &:=  \frac{2M_{\mathrm{ker}}}{\alpha}
    \frac{\phi_m(\eta_m)}{\eta_m}
    \sum_{\substack{m\le j\le n\\x_j\le\eta_m}}
    G_{n,j}\tau_j,
    \label{eq:J}\\
    \mathcal{J}_m^{(2)}(x_n)
    &:=  \frac{2M_{\mathrm{ker}}}{\alpha}
    \sum_{\substack{m\le j\le n\\x_{j-1}\ge\eta_m}}
    G_{n,j}  \int_{x_{j-1}}^{x_j}(-\phi_m'(s))\,\dd s.
    \label{eq:P}
\end{align}

Combining \eqref{eq:Wmnfrac},
\eqref{eq:weight-error-bound}, and \eqref{eq:WmnL1exp},
we obtain
\begin{equation}\label{eq:DLamExp}
    (\mathcal D_\tau^\alpha\mathcal W_m)^n +\lambda\mathcal W_m^n
    \ge  \mathcal H_m(x_n)-\kappa_m\phi_m'(x_n)  - E_m(x_n)   -\mathcal{J}_m^{(1)}(x_n)-\mathcal{J}_m^{(2)}(x_n).
\end{equation}
We first estimate $\mathcal{J}_m^{(1)}$.
All sums below are restricted to $m\le j\le n$. Suppose first that $\eta_m<x_n \le2\eta_m$.
By monotonicity and a shift by $\tau$,
the left-endpoint quadrature error satisfies
\begin{equation}\label{eq:sumGnjtauj}
    \sum_{x_j\le\eta_m}G_{n,j}\tau_j
    \le
    \int_{\eta_m-\tau}^{\eta_m}
    g_{1-\alpha}(x_n-s)\,\dd s
    \le C_\alpha\tau^{1-\alpha}.
\end{equation}
Combining this with \eqref{eq:J},
\eqref{eq:Hmlower2}, and $x_n\le2\eta_m$ gives
\begin{equation}\label{eq:Jnearratio}
    \frac{\mathcal{J}_m^{(1)}(x_n)}{\mathcal H_m(x_n)}
    \le C_\alpha M_{\mathrm{ker}}
    \left(\frac{\tau}{\eta_m}\right)^{1-\alpha}.
\end{equation}
For $x_n\ge2\eta_m$, the derivative of
$s\mapsto g_{1-\alpha}(x_n-s)$ is bounded by
$C_\alpha x_n^{-1-\alpha}$ on $0\le s\le\eta_m$.
Hence,
\begin{equation}
    \sum_{x_j\le\eta_m}G_{n,j}\tau_j \le C_\alpha x_n^{-1-\alpha}
    \sum_{x_j\le\eta_m}\tau_j^2 \le C_\alpha\tau\eta_m x_n^{-1-\alpha}.
\end{equation}
Using \eqref{eq:J} and \eqref{eq:Hmlower2} again yields
\begin{equation}\label{eq:Joldratio}
    \frac{\mathcal{J}_m^{(1)}(x_n)}{\mathcal H_m(x_n)}
    \le C_\alpha M_{\mathrm{ker}}   \frac{\tau\phi_m(\eta_m)x_n^{-1-\alpha}}   {\eta_m\phi_m(\eta_m)x_n^{-1-\alpha}}
    =  C_\alpha M_{\mathrm{ker}} \frac{\tau}{\eta_m}.
\end{equation}
Since $\tau/\eta_m\le1$, the two cases imply
\begin{equation}\label{eq:J1upp}
    \mathcal{J}_m^{(1)}(x_n)
    \le C_\alpha M_{\mathrm{ker}}   \left(\frac{\tau}{\eta_m}\right)^{1-\alpha}
    \mathcal H_m(x_n),
    \qquad x_n>\eta_m.
\end{equation}

We next estimate $\mathcal{J}_m^{(2)}$.
Splitting the sum according to whether $x_{j-1}<x_n/2$
or $x_{j-1}\ge x_n/2$, we write
\begin{equation}
    \mathcal{J}_m^{(2)}(x_n)
    = \mathcal{J}_m^{(2,1)}(x_n) +  \mathcal{J}_m^{(2,2)}(x_n),
\end{equation}
where
\begin{align}
    \mathcal{J}_m^{(2,1)}(x_n)
    &:=  \frac{2M_{\mathrm{ker}}}{\alpha}   \sum_{\eta_m\le x_{j-1}<x_n/2}
    G_{n,j}  \int_{x_{j-1}}^{x_j}(-\phi_m'(s))\,\dd s,  \\
    \mathcal{J}_m^{(2,2)}(x_n)
    &:=  \frac{2M_{\mathrm{ker}}}{\alpha}
    \sum_{x_{j-1}\ge\max\{\eta_m,t/2\}}  G_{n,j}   \int_{x_{j-1}}^{x_j}(-\phi_m'(s))\,\dd s.
\end{align}

For the first part, since $x_n>\eta_m\ge4\tau$, we have $x_j\le x_n/2+\tau\le 3x_n/4$.
The same kernel derivative estimate as above therefore gives $G_{n,j}\le C_\alpha\tau x_n^{-1-\alpha}$. Combining this with \eqref{eq:Hmlower2}, we obtain
\begin{equation}\label{eq:Pold}
    \begin{aligned}
        \mathcal{J}_m^{(2,1)}(x_n)
        &\le C_\alpha M_{\mathrm{ker}}\tau x_n^{-1-\alpha}  \int_{\eta_m}^{\infty}(-\phi_m'(s))\,\dd s\\
        &=   C_\alpha M_{\mathrm{ker}}  \tau\phi_m(\eta_m)x_n^{-1-\alpha}
        \le C_\alpha M_{\mathrm{ker}}  \frac{\tau}{\eta_m}\mathcal H_m(x_n).
    \end{aligned}
\end{equation}

 For each cell in the second part, every $s\in[x_{j-1},x_j]$ satisfies $\max\{\eta_m, x_n/2\}\le s\le x_n$.
By \eqref{eq:phiasymp} with $p=1$, we have $-\phi_m'(s)\le C_\alpha(-\phi_m'(t))$.
Applying the shift argument in \eqref{eq:sumGnjtauj} to the whole interval $[0,t]$, we obtain
\begin{equation}
    \mathcal{J}_m^{(2,2)}(x_n)
    \le C_\alpha M_{\mathrm{ker}}(-\phi_m'(x_n))  \sum_{j=m}^nG_{n,j}\tau_j
    \le C_\alpha M_{\mathrm{ker}}  \tau^{1-\alpha}(-\phi_m'(x_n)).
\end{equation}
Combining the two parts gives
\begin{equation}\label{eq:J2upp}
    \mathcal{J}_m^{(2)}(x_n)
    \le  C_\alpha M_{\mathrm{ker}}  \left(  \frac{\tau}{\eta_m}\mathcal H_m(x_n)  +  \tau^{1-\alpha}(-\phi_m'(x_n))   \right),
    \qquad x_n>\eta_m.
\end{equation}
We now choose $D_*$ and $L_*$ to absorb the interpolation
and weight errors.
Under the scale condition $D_*L_*^{-(1-\alpha)}\le1$,
combining \eqref{eq:Emupper}, \eqref{eq:J1upp},
and \eqref{eq:J2upp} with
$\tau/\eta_m\le L_*^{-1}$ gives for $x_n>\eta_m$ that
\begin{equation}
    E_m(x_n) +\mathcal{J}_m^{(1)}(x_n) +\mathcal{J}_m^{(2)}(x_n)
    \le   \delta(L_*)\mathcal H_m(x_n)  +C_\alpha M_{\mathrm{ker}}  \tau^{1-\alpha}(-\phi_m'(x_n)),
\end{equation}
where
\begin{equation}
    \delta(L_*) := C_\alpha\left(   L_*^{-(2-\alpha)}  +M_{\mathrm{ker}}L_*^{-(1-\alpha)}   \right),
    \qquad L_*\ge1.
\end{equation}
Substituting this estimate into \eqref{eq:DLamExp}
and using $\kappa_m=D_*\tau^{1-\alpha}$, we obtain
\begin{equation}\label{eq:DLconstraint}
    (\mathcal D_\tau^\alpha\mathcal W_m)^n  +\lambda\mathcal W_m^n
    \ge  \left(1-\delta(L_*)\right)\mathcal H_m(x_n)
    +  \left(D_*-C_\alpha M_{\mathrm{ker}}\right)  \tau^{1-\alpha}(-\phi_m'(x_n)).
\end{equation}
First choose
\begin{equation}\label{eq:Dchoice}
    D_*:=\max\{1,\,2C_\alpha M_{\mathrm{ker}}\},
\end{equation}
so that $D_*-C_\alpha M_{\mathrm{ker}}\ge D_*/2$.
With $D_*$ fixed, choose $L_*\ge4$ sufficiently large that
\begin{equation}\label{eq:Lstar}
    D_*L_*^{-(1-\alpha)}\le1,
    \qquad
    \delta(L_*)\le\frac12.
\end{equation}
Such a choice is possible since $0<\alpha<1$.
Consequently, \eqref{eq:DLconstraint} yields
\begin{equation}\label{eq:posres}
    (\mathcal D_\tau^\alpha\mathcal W_m)^n
    +\lambda\mathcal W_m^n
    \ge
    \frac12\mathcal H_m(x_n)
    +\frac12\kappa_m(-\phi_m'(x_n))
    \ge
    \frac12\mathcal H_m(x_n),
    \qquad x_n>\eta_m.
\end{equation}
The chosen constants $D_*$ and $L_*$
depend only on $\alpha$ and $M_{\mathrm{ker}}$,
and are independent of $\lambda$, the time steps,
and the final time.

\subsection{Supersolution estimates}
\label{subsec:step5}

For $m\le n$ with $x_n\le\eta_m$, the definition of the linear part
of $\mathcal W_m$ and the positivity of the weights give
\begin{equation}\label{eq:linearres}
    (\mathcal D_\tau^\alpha\mathcal W_m)^n
    +\lambda\mathcal W_m^n
    =
    \frac{\phi_m(\eta_m)}{\eta_m}
    \left(
        \sum_{j=m}^n c_{n-j}^n\tau_j+\lambda x_n
    \right)>0.
\end{equation}
Combining this with \eqref{eq:posres} for $x_n>\eta_m$,
we obtain
\begin{equation}\label{eq:OpWnonneg}
    (\mathcal D_\tau^\alpha\mathcal W_m)^n
    +\lambda\mathcal W_m^n
    \ge0,
    \qquad m\le n\le N.
\end{equation}

At the source node $n=m$, since $x_m=\tau_m\le\eta_m$
and $\mathcal W_m^j=0$ for $j<m$, we have
\begin{equation}\label{eq:DWasym}
    \tau_m\left(
        (\mathcal D_\tau^\alpha\mathcal W_m)^m
        +\lambda\mathcal W_m^m
    \right)
    = \frac{\tau_m^2(c_0^m+\lambda)\phi_m(\eta_m)}
         {\eta_m}.
\end{equation}
Using \eqref{eq:c0mest} and $\kappa_m\eta_m^{\alpha-1}\le1$,
combining the lower bound in \eqref{eq:phiasymp}
for $p=0$ with the step-size ratio estimate
and the diagonal weight lower bound gives
\begin{equation}\label{eq:c_lambda_phi}
    \frac{\tau_m^2(c_0^m+\lambda)\phi_m(\eta_m)}
         {\eta_m}
    \ge
    c_\alpha\lambda
    \frac{\tau_m^2(c_0^m+\lambda)\eta_m^{\alpha-2}}
         {(1+\lambda\eta_m^\alpha)^2}
    \ge
    \gamma_*\frac{\lambda}{1+\lambda\tau_m^\alpha},
\end{equation}
where $\gamma_*$ is introduced in \eqref{eq:sourceconstant}.
Finally, using $\lambda\tau_m^\alpha\le1$,
we conclude from \eqref{eq:DWasym}
and \eqref{eq:c_lambda_phi} that
\begin{equation}\label{eq:sourcehalf}
 (\mathcal D_\tau^\alpha\mathcal V_m)^m
    +\lambda\mathcal V_m^m=    \frac{2\tau_m}{\gamma_*}\left(
        (\mathcal D_\tau^\alpha\mathcal W_m)^m
        +\lambda\mathcal W_m^m
    \right)
    \ge \lambda.
\end{equation}
For $n>m$, the residual is nonnegative
by \eqref{eq:OpWnonneg}.
Thus
\begin{equation}\label{eq:sourcegoal}
    (\mathcal D_\tau^\alpha\mathcal V_m)^n
    +\lambda\mathcal V_m^n
    \ge\lambda\delta_{nm},
    \qquad 1\le n\le N.
\end{equation}

We next estimate the size of $\mathcal V_m^n$.
For $x_n\ge\eta_m$, we have
$\kappa_m x_n^{\alpha-1}
\le\kappa_m\eta_m^{\alpha-1}\le1$.
Applying \eqref{eq:phiasymp} with $p=0$
and using \eqref{eq:rlambdaasy}, we obtain
\begin{equation}
    \mathcal W_m(x_n)
    =\phi_m(x_n)  \le C_\alpha r_\lambda(x_n).
\end{equation}
For $0< x_n <\eta_m$, the definition of the linear part
and the monotonicity of $r_\lambda$ give
\begin{equation}\label{eq:Wcomp}
    \mathcal W_m(x_n)
    = \frac{x_n}{\eta_m}\phi_m(\eta_m)
    \le C_\alpha\frac{x_n}{\eta_m}r_\lambda(\eta_m)
    \le C_\alpha r_\lambda(x_n).
\end{equation}
Hence $\mathcal W_m(x_n)\le C_\alpha r_\lambda(x_n)$
for all $x_n>0$.
Using the monotonicity of $r_\lambda$ again
and the definition of $\bar{u}_m$, we obtain, for $n\ge m$,
\begin{equation}\label{eq:Vsize}
    \mathcal V_m^n
    \le
    \frac{2C_\alpha}{\gamma_*}\tau_m r_\lambda(x_n)
    \le
    \frac{2C_\alpha}{\gamma_*}
    \int_{x_n-\tau_m}^{x_n}r_\lambda(s)\,\dd s
    =
    \frac{2C_\alpha}{\gamma_*}\bar{u}_m(t_n).
\end{equation}
Since $r_\lambda$ is integrable near zero,
this estimate also covers $n=m$.
Taking $C_{\mathrm{bar}}:=2C_\alpha/\gamma_*$
proves \eqref{eq:Vbound}.
This constant depends only on
$\alpha,K,\nu_{\mathrm{ker}},M_{\mathrm{ker}}$,
and is independent of $\lambda$, the time grid,
and the final time.
Together with \eqref{eq:sourcegoal},
this completes the proof of Lemma~\ref{lem:V}
and hence that of Theorem~\ref{thm:mainest}.

\section{Standard discretization schemes}\label{sec:standard-schemes}

Besides the L1 scheme, we verify that our results in this paper apply to other typical schemes on uniform meshes. In particular, we verify that Assumption~\ref{ass:proofcond} holds for some typical complete-monotonicity preserving
schemes (and thus completely positive) on uniform-mesh \cite{liwang2021cm}. We will consider the uniform-step Gr\"unwald--Letnikov, Li--Liu deconvolution,
and $\theta$-method convolution quadrature schemes; see
\cite{paper2,liliu2019caputo,liwang2021cm}.

\subsection{The Gr\"unwald--Letnikov scheme}
\label{subsec:schemes-GL}

On the uniform grid $t_n=n\tau$, let
$\omega_k:=(-1)^k\binom{\alpha}{k}$.
The unshifted Gr\"unwald--Letnikov approximation of the
Caputo derivative can be written as
\begin{equation}\label{eq:GL-Caputo}
    \cD_{\mathrm{GL},\tau}^{\alpha}v_n
    :=\tau^{-\alpha}\sum_{k=0}^{n}
    \omega_k(v_{n-k}-v_0)
    =\sum_{j=1}^{n}
    c_{n-j}^{\mathrm{GL}}(v_j-v_{j-1}),
\end{equation}
where
\begin{equation}\label{eq:GL-increment-weights}
    c_k^{\mathrm{GL}}
    =\tau^{-\alpha}\sum_{q=0}^{k}\omega_q
    =\tau^{-\alpha}
      \frac{\Gamma(k+1-\alpha)}
           {\Gamma(1-\alpha)\Gamma(k+1)},
    \qquad k\ge0.
\end{equation}
Since
\begin{equation}
    c_0^{\mathrm{GL}}=\tau^{-\alpha}>0,
    \qquad
    \frac{c_{k+1}^{\mathrm{GL}}}{c_k^{\mathrm{GL}}}
    =\frac{k+1-\alpha}{k+1}\in(0,1),
\end{equation}
condition~\ref{cond1} follows from \eqref{eq:b_and_c}.
Condition~\ref{cond2} holds with $K=1$.
To verify condition~\ref{cond3}, set $k=n-j$.
For $k\ge1$, Gautschi's inequality gives
\begin{equation}\label{eq:cGLcontrol}
    g_{1-\alpha}((k+1)\tau)
    \le c_k^{\mathrm{GL}}
    \le g_{1-\alpha}(k\tau).
\end{equation}
Since $\ell_{n,j}$ is the average of the decreasing function
$g_{1-\alpha}$ over $[k\tau,(k+1)\tau]$,
it lies between the same endpoints. Hence,
\begin{equation}
    \frac{c_k^{\mathrm{GL}}}{\ell_{n,j}}
    \ge\left(\frac{k}{k+1}\right)^\alpha
    \ge\frac{k}{k+1}
    \ge\frac12,
\end{equation}
and
\begin{equation}\label{eq:GL-relative-deviation}
    \frac{|c_k^{\mathrm{GL}}-\ell_{n,j}|}{\ell_{n,j}}
    \le
    \frac{g_{1-\alpha}(k\tau)
          -g_{1-\alpha}((k+1)\tau)}
         {g_{1-\alpha}((k+1)\tau)}
    =\left(1+\frac1k\right)^\alpha-1
    \le\frac1k
    \le\frac{2}{k+1}.
\end{equation}
For $k=0$, another application of Gautschi's inequality yields
\begin{equation}
    \frac12
    \le2^{-\alpha}
    \le\frac{c_0^{\mathrm{GL}}}{\ell_{n,n}}
    =\Gamma(2-\alpha)
    \le1.
\end{equation}
Thus, the lower bound
$c_k^{\mathrm{GL}}\ge\ell_{n,j}/2$
and the relative deviation bound
\eqref{eq:GL-relative-deviation} also hold for $k=0$.
Since $\tau_j/(t_n-t_{j-1})=1/(k+1)$,
these estimates verify condition~\ref{cond3}.
Therefore, Assumption~\ref{ass:proofcond} holds.

\subsection{A deconvolution scheme}\label{subsec:schemes-LL}

On the uniform grid, the scheme of Li and Liu~\cite{liliu2019caputo}
is specified by the inverse kernel $A=B^{-1}$:
\begin{equation}\label{eq:LL-integral-weights}
 a_k^{\mathrm{LL}}
 =\frac{\tau^\alpha}{\Gamma(1+\alpha)}
       ((k+1)^\alpha-k^\alpha),\qquad k\ge0.
\end{equation}
Thus $v_n-v_0=\sum_{j=1}^n a_{n-j}^{\mathrm{LL}}F_j$.
Writing $b^{\mathrm{LL}}=(a^{\mathrm{LL}})^{(-1)}$ for the
convolution inverse, the corresponding derivative and increment
weights are
\[
 \cD_{\mathrm{LL},\tau}^\alpha v_n
 =\sum_{k=0}^n b_k^{\mathrm{LL}}(v_{n-k}-v_0),\qquad
 c_k^{\mathrm{LL}}=\sum_{q=0}^k b_q^{\mathrm{LL}}.
\]
Proposition~3.1 of \cite{liliu2019caputo} gives
$b_0^{\mathrm{LL}}>0$, $b_k^{\mathrm{LL}}<0$ for $k\ge1$, and
$\sum_{k\ge0}b_k^{\mathrm{LL}}=0$; hence
$c_k^{\mathrm{LL}}=-\sum_{q>k}b_q^{\mathrm{LL}}>0$, verifying
complete positivity. To obtain the quantitative comparison with
L1, put $\widehat a(z)=\sum_{k\ge0}a_k^{\mathrm{LL}}z^k$.
Using $\operatorname{Li}_{-\alpha}(z)=\sum_{m\ge1}m^\alpha z^m$,
\[
 \widehat a(z)=\frac{\tau^\alpha(1-z)}{\Gamma(1+\alpha)z}
                  \operatorname{Li}_{-\alpha}(z),\qquad
 \sum_{k\ge0}c_k^{\mathrm{LL}}z^k
       =\frac1{(1-z)\widehat a(z)}.
\]
The polylogarithm expansion at $z=1$, with $x=1-z$, gives
\[
 \widehat a(z)=\tau^\alpha x^{-\alpha}
       \left(1+\frac{1-\alpha}{2}x+O_\alpha(x^{1+\alpha})\right).
\]
Inverting and applying coefficient transfer at $z=1$ yields
$c_k^{\mathrm{LL}}=\tau^{-\alpha}k^{-\alpha}
(1+O_\alpha(k^{-1}))/\Gamma(1-\alpha)$.
Since $\ell_{n,j}$ has the same leading term for $k=n-j$, and
all $c_k^{\mathrm{LL}}$ are positive, there exist
$\nu_\alpha>0$ and $M_\alpha<\infty$ such that, for every $k\ge0$,
\[
 c_k^{\mathrm{LL}}\ge\nu_\alpha\ell_{n,j},\qquad
 \left|c_k^{\mathrm{LL}}/\ell_{n,j}-1\right|
       \le\frac{M_\alpha}{k+1}.
\]
The finitely many initial indices are included by adjusting the
constants. Therefore Assumption~\ref{ass:proofcond} holds with
$K=1$, $\nu_{\mathrm{ker}}=\nu_\alpha$, and
$M_{\mathrm{ker}}=M_\alpha$; Theorem~\ref{thm:mainest} applies
under \eqref{eq:stepsizeconstraint}.

\subsection{The \texorpdfstring{$\theta$}{theta}-method convolution quadrature}
\label{subsec:schemes-BE}

On the uniform grid, fix $\theta\ge1$. The generating function
for fractional integration and its inverse are
\begin{equation}\label{eq:theta-generating-function}
 \sum_{k\ge0}a_k^\theta z^k
 =\tau^\alpha\left(\frac{\theta+(1-\theta)z}{1-z}\right)^\alpha,
 \qquad
 \sum_{k\ge0}b_k^\theta z^k
 =\tau^{-\alpha}\left(\frac{1-z}{\theta+(1-\theta)z}\right)^\alpha.
\end{equation}
Thus the Caputo approximation and its increment weights are
\begin{equation}\label{eq:theta-Caputo}
 \cD_{\theta,\tau}^\alpha v_n
 :=\sum_{k=0}^n b_k^\theta(v_{n-k}-v_0)
 =\sum_{j=1}^n c_{n-j}^\theta(v_j-v_{j-1}),\qquad
 \sum_{k\ge0}c_k^\theta z^k
 =\tau^{-\alpha}\frac{(1-z)^{\alpha-1}}
                         {(\theta+(1-\theta)z)^\alpha}.
\end{equation}
The case $\theta=1$ is precisely backward Euler convolution
quadrature, and hence coincides with the GL scheme above.
For general $\theta\ge1$, Euler's beta integral gives the useful
moment representation
\begin{equation}\label{eq:theta-moments}
 c_k^\theta
 =\tau^{-\alpha}\theta^{-\alpha}\frac{\sin(\pi\alpha)}\pi
   \int_0^1(1-v/\theta)^k v^{\alpha-1}(1-v)^{-\alpha}\,dv,
 \qquad k\ge0.
\end{equation}
Indeed, summing the geometric series inside this integral recovers
the generating function in \eqref{eq:theta-Caputo}.
Since $0\le1-v/\theta\le1$, the weights $c_k^\theta$ are positive
and decreasing. Consequently $b_0^\theta=c_0^\theta>0$,
$b_k^\theta=c_k^\theta-c_{k-1}^\theta\le0$ for $k\ge1$, and
$\sum_{j=0}^k b_j^\theta=c_k^\theta>0$. This verifies complete
positivity directly, consistently with \cite{liwang2021cm}.
Moreover, $1-v/\theta\ge1-v$ implies
\[
 c_k^\theta\ge\theta^{-\alpha}c_k^{\mathrm{GL}}
       \ge\tfrac12\theta^{-\alpha}\ell_{n,j},\qquad k=n-j.
\]
For fixed $\theta$, expansion at $z=1$ gives
\[
 (1-z)^{\alpha-1}(1+(\theta-1)(1-z))^{-\alpha}
 =(1-z)^{\alpha-1}
  -\alpha(\theta-1)(1-z)^\alpha
  +O_{\alpha,\theta}((1-z)^{\alpha+1}).
\]
Using the binomial coefficients (or expanding
\eqref{eq:theta-moments} at $v=0$), it follows that
\[
 c_k^\theta=\frac{\tau^{-\alpha}k^{-\alpha}}{\Gamma(1-\alpha)}
                 (1+O_{\alpha,\theta}(k^{-1})),\qquad
 \ell_{n,j}=\frac{\tau^{-\alpha}k^{-\alpha}}{\Gamma(1-\alpha)}
                 (1+O_\alpha(k^{-1})).
\]
Thus $|c_k^\theta/\ell_{n,j}-1|\le C_{\alpha,\theta}/(k+1)$
for all $k\ge0$, enlarging the constant for finitely many indices.
Assumption~\ref{ass:proofcond} holds with
$K=1$, $\nu_{\mathrm{ker}}=\theta^{-\alpha}/2$, and
$M_{\mathrm{ker}}=C_{\alpha,\theta}$, so Theorem~\ref{thm:mainest}
applies under \eqref{eq:stepsizeconstraint}. The constants need
not be uniform as $\theta\to\infty$; already
$c_0^\theta/\ell_{n,n}=\theta^{-\alpha}\Gamma(2-\alpha)\to0$.
The restriction on $\theta$ matters: for $\theta>0$,
\[
 b_2^\theta=\tau^{-\alpha}
       \frac{\alpha(1+\alpha-2\theta)}{2\theta^{\alpha+2}},
\]
which is positive when $\theta<(1+\alpha)/2$ and then violates
complete positivity. In particular, the trapezoidal choice
$\theta=1/2$ is excluded for every $0<\alpha<1$.

\section[Decay-preserving estimate for linear subdiffusion models]{Decay-preserving estimate for linear\texorpdfstring{\\}{ }subdiffusion models}
\label{sec:linear-subdiffusion}

\subsection{Model problem and parameter-dependent convergence}

Consider the linear reaction--subdiffusion problem
\begin{equation}\label{eq:subdiffusion-model}
\begin{aligned}
 D_0^\alpha u-\Delta u&=\kappa u+f,
       &&x\in\Omega,\quad t>0,\\
 u&=0,&&x\in\partial\Omega,\quad t>0,\\
 u(x,0)&=u_0(x),&&x\in\Omega,
\end{aligned}
\end{equation}
where $0<\alpha<1$ and $\Omega$ is a bounded domain.
Throughout this section, $\|\cdot\|$ and $\langle\cdot,\cdot\rangle$
denote the norm and inner product in $L^2(\Omega)$.
Let $\lambda_1>0$ be the first Dirichlet eigenvalue of $-\Delta$,
so that
\[
 \langle-\Delta v,v\rangle\ge\lambda_1\|v\|^2,
 \qquad v\in H^2(\Omega)\cap H_0^1(\Omega).
\]
Write $A=-\Delta-\kappa I$ for the self-adjoint Dirichlet
realization and assume the dissipative gap
$\lambda:=\lambda_1-\kappa>0$. Its domain is
$D(A)=\{v\in H_0^1(\Omega):\Delta v\in L^2(\Omega)\}$,
which equals $H^2(\Omega)\cap H_0^1(\Omega)$ on a bounded
$C^{1,1}$ domain or a convex polygon.

Zhang, Zhang, and Zhao~\cite{zhangzhangzhao-decay} report a
parameter-dependent convergence phenomenon for the L1 method.
For the manufactured solution $u(x,t)=t^\alpha\sin(\pi x/L)$
on $\Omega=(0,L)$, their Table~1 and Tables~10--12 show apparent
orders between $1$ and $2-\alpha$ at a fixed final time.
For example, with $\alpha=1/2$, $L=\pi$, $\kappa=0$, and $T=1$,
the reported orders decrease from $1.06$ to $1.02$ as $N$ increases
from $64$ to $512$. With $L=1$, $\kappa=-8$, and $T=10$, they
instead decrease from $1.45$ to $1.37$ over the same refinements.
Thus, a rate close to $2-\alpha=3/2$ can be observed on practical
meshes, although further refinement tends to reveal first-order
behavior. Smaller domains and more negative reaction coefficients
increase the dissipative gap $\lambda_1-\kappa$; on an interval,
$\lambda_1=\pi^2/L^2$. Increasing the final time also favors the
higher observed rates in the reported experiments. Their Table~13
reports rates approaching one when $\kappa\ge\lambda_1$.

Conjecture~4.1 of \cite{zhangzhangzhao-decay} motivates a
bound separating first-order and $(2-\alpha)$-order contributions.
We distinguish fixed-time estimates for general smooth forcing
from estimates uniform for all positive times when the forcing
approaches a stationary limit. The analysis concerns temporal
error; it also applies to a fixed positive self-adjoint spatial
discretization, with its smallest eigenvalue replacing $\lambda_1$.

\subsection{Finite time estimates for general smooth forcing}
\label{subsec:subdiffusion-finite}

Fix $0<T\le T_*$, where $T_*=O(1)$ is independent of the mesh.
For general smooth $f$ we assume $u\in C([0,T];L^2(\Omega))$,
$u\in C^2((0,T];L^2(\Omega))$, and
\begin{equation}\label{eq:subdiffusion-finite-regularity}
 \|\partial_t^k u(t)\|\le M_T t^{\alpha-k},
 \qquad k=1,2,\quad 0<t\le T.
\end{equation}
A sufficient data condition is $u_0\in D(A)$ and
$f\in C^2([0,T];L^2(\Omega))$, with
\[
 M_T\le C_\alpha\bigl(\|Au_0-f(0)\|
       +T\|f'\|_{L^\infty(0,T;L^2)}
       +T^2\|f''\|_{L^\infty(0,T;L^2)}\bigr).
\]
Indeed, writing $S_A(t)=E_\alpha(-t^\alpha A)$ and
$K_A(t)=A^{-1}r_A(t)$, differentiation of the solution formula gives
\[
 u'=-K_A(t)(Au_0-f(0))+K_A*f',\qquad
 u''=-K_A'(t)(Au_0-f(0))+K_A(t)f'(0)+K_A*f''.
\]
The bounds $\|K_A^{(j)}(t)\|\le C_\alpha t^{\alpha-1-j}$,
$j=0,1$, follow from Lemma~\ref{d:lem:ml} and complete
monotonicity, as detailed below in the proof of
Proposition~\ref{prop:subdiffusion-regularity}.
No compatibility condition $Au_0=f(0)$ is imposed.

Write $\tau=\max_j\tau_j$ and suppose that
$\tau/\min_j\tau_j\le K$. Let $u_n$ solve the L1 scheme
\begin{equation}\label{eq:subdiffusion-L1}
 \cD_\tau^\alpha u_n-\Delta u_n=\kappa u_n+f(t_n),
 \qquad 0<t_n\le T,
\end{equation}
with the same homogeneous boundary condition and a prescribed numerical
initial value, where
$\cD_\tau^\alpha$ is given by \eqref{eq:L1scheme}.
Define the temporal error and consistency error by
\begin{equation}\label{eq:subdiffusion-errors}
 e_n:=u_n-u(t_n),\qquad
 \eta_n:=\cD_\tau^\alpha u(t_n)-D_0^\alpha u(t_n).
\end{equation}
Here $\cD_\tau^\alpha u(t_n)$ means that the L1 operator is applied
to the sequence of exact nodal values.

\begin{lemma}\label{lem:subdiffusion-consistency-finite}
Under \eqref{eq:subdiffusion-finite-regularity}, there is a constant
$C=C(\alpha,K)$ such that
\begin{equation}\label{eq:subdiffusion-consistency-split-finite}
 \|\eta_n\|\le C M_T\left[
       \left(\frac\tau{t_n}\right)^{1+\alpha}
       +\left(\frac{\tau}{t_n}\right)^{2-\alpha}\right],\qquad 0<t_n\le T.
\end{equation}
The constant is independent of $\lambda$, $n$, and $T$;
all dependence on the data and time horizon is in $M_T$.
\end{lemma}

\begin{proof}
Let $I_\tau u$ be the piecewise linear interpolant and set
$w=I_\tau u-u$. Integration by parts, using $w(t_j)=0$, gives
\[
 \eta_n=-\frac\alpha{\Gamma(1-\alpha)}
       \int_0^{t_n}(t_n-s)^{-1-\alpha}w(s)\,ds.
\]
The endpoint term at $t_n$ vanishes since $w(s)=O(t_n-s)$ there.
On the first interval, \eqref{eq:subdiffusion-finite-regularity} implies
$\|w(s)\|\le C M_T\tau_1^\alpha$. On every later interval,
linear interpolation gives
\[
 \|w(s)\|\le C M_T\tau_j^2t_{j-1}^{\alpha-2},
 \qquad s\in[t_{j-1},t_j],\quad j\ge2.
\]

If $t_n<4\tau$, there are at most $4K$ intervals. The increment
bound $\|u(t_j)-u(t_{j-1})\|\le C M_T(t_j^\alpha-t_{j-1}^\alpha)$
and the L1 weights give $\|\cD_\tau^\alpha u(t_n)\|\le C M_T$;
the Caputo integral likewise gives $\|D_0^\alpha u(t_n)\|\le C M_T$.
Since $t_n<4\tau$, the first term on the right of
\eqref{eq:subdiffusion-consistency-split-finite} alone controls this bound.

Now, assume $t_n\ge4\tau$. Since adjacent positive grid
points are comparable with constants depending on $K$, these
estimates yield
\[
 \int_0^{t_n/2}\|w(s)\|\,ds
 \le C M_T\left(\tau_1^{1+\alpha}
        +\tau^2\int_{\tau_1}^{t_n}s^{\alpha-2}\,ds\right)
 \le C M_T\tau^{1+\alpha}.
\]
For $s\ge t_n/2$, the interpolation interval lies in
$[t_n/4,t_n]$, and hence
\[
 \|w(s)\|\le C M_T t_n^{\alpha-2}
                     \min\{\tau^2,\tau(t_n-s)\}.
\]
For the second bound inside the minimum, use the endpoint
interpolation bound on the last interval; on earlier intervals,
$t_n-s\ge\tau_n\ge\tau/K$ suffices.
Splitting the integral at $t_n/2$ therefore gives
\begin{align*}
 \|\eta_n\|
 &\le C M_T\tau^{1+\alpha}t_n^{-1-\alpha}
   +C M_T t_n^{\alpha-2}
       \int_0^{t_n/2}r^{-1-\alpha}\min\{\tau^2,\tau r\}\,dr\\
 &\le C M_T\bigl(\tau^{1+\alpha}t_n^{-1-\alpha}
                       +\tau^{2-\alpha}t_n^{\alpha-2}\bigr).
\end{align*}

\end{proof}

\begin{theorem}\label{thm:subdiffusion-decay-finite}
Suppose that \eqref{eq:subdiffusion-finite-regularity} holds,
$\lambda=\lambda_1-\kappa>0$, $\tau/\min_j\tau_j\le K$, and
$\lambda\tau^\alpha\le1$. Then
\begin{equation}\label{eq:subdiffusion-decay-finite}
\begin{aligned}
 \|e_n\|\le{}&\|e_0\|E_\alpha(-\rho^{-1}\lambda t_n^\alpha)\\
 &+C M_T\left[
   \frac{\tau t_n^{\alpha-1}}{(1+\lambda t_n^\alpha)^2}
   +\frac{\tau^{2-\alpha}t_n^{2\alpha-2}}
          {1+\lambda t_n^\alpha}\right],
 \qquad 0<t_n\le T,
\end{aligned}
\end{equation}
where $C$ depends only on $\alpha$ and $K$, and $\rho$ is the
constant in \eqref{eq:discrete-gronwall} for L1.
\end{theorem}

\begin{proof}
Subtracting the equation for the exact nodal values from
\eqref{eq:subdiffusion-L1} gives
\begin{equation}\label{eq:subdiffusion-error-equation}
 \cD_\tau^\alpha e_n-\Delta e_n=\kappa e_n-\eta_n.
\end{equation}
The L1 weights increase with the interval index $j$.
The subgradient inequality for the norm, applied to the resulting
nonnegative combination of differences $e_n-e_j$, gives, when
$e_n\ne0$,
\[
 \cD_\tau^\alpha\|e_n\|
 \le\left\langle\frac{e_n}{\|e_n\|},
                         \cD_\tau^\alpha e_n\right\rangle
 \le-\lambda\|e_n\|+\|\eta_n\|.
\]
When $e_n=0$, the same scalar inequality follows directly from
the signs of the L1 coefficients: its left-hand side is
nonpositive. Thus \eqref{eq:discrete-gronwall} applies. Introduce
\[
 k_\lambda(t):=\lambda^{-1}r_\lambda(t)
   =\alpha t^{\alpha-1}E_\alpha'(-\lambda t^\alpha),\qquad
 H_\lambda(t):=\int_0^t k_\lambda(s)\,ds
   =\frac{1-E_\alpha(-\lambda t^\alpha)}\lambda.
\]
Keeping the factor $\lambda^{-1}$ in
\eqref{eq:discrete-gronwall}, we obtain
\begin{equation}\label{eq:subdiffusion-error-convolution}
 \|e_n\|\le\|e_0\|E_\alpha(-\rho^{-1}\lambda t_n^\alpha)
  +C\sum_{m=1}^n\|\eta_m\|
         \int_{t_{m-1}}^{t_m}k_\lambda(t_n-s)\,ds.
\end{equation}
By Lemma~\ref{d:lem:ml},
\begin{equation}\label{eq:subdiffusion-response-bounds}
 k_\lambda(t)\asymp_\alpha
           \frac{t^{\alpha-1}}{(1+\lambda t^\alpha)^2},
 \qquad
 H_\lambda(t)\le C_\alpha\frac{t^\alpha}{1+\lambda t^\alpha}.
\end{equation}
Indeed, integration of the kernel bound gives
$H_\lambda(t)\le C_\alpha t^\alpha$, while positivity of
$E_\alpha(-\lambda t^\alpha)$ gives $H_\lambda(t)\le\lambda^{-1}$.

Set $a=1+\alpha$, $b=2-\alpha$, and define
\[
 h(s):=\left(\frac\tau{s+\tau}\right)^a
       +(\tau/(s+\tau))^b.
\]
For $s\in[t_{m-1},t_m]$, we have
$s+\tau\le(1+K)t_m$ and $s\le t_m$. Thus Lemma~\ref{lem:subdiffusion-consistency-finite}
gives $\|\eta_m\|\le C M_T h(s)$. Both powers exceed one, so
$\int_0^\infty h(s)\,ds\le C_\alpha\tau$. Splitting at $t/2$,
using \eqref{eq:subdiffusion-response-bounds} and the monotonicity
of $h$, yields, for every $0<t=t_n\le T$,
\begin{equation}\label{eq:subdiffusion-convolution-split}
 \int_0^t h(s)k_\lambda(t-s)\,ds
 \le C\tau k_\lambda(t)+C\left[
 \tau^a t^{-a}+\tau^b t^{-b}
 \right]H_\lambda(t).
\end{equation}
Indeed, on the first half $k_\lambda(t-s)\le C_\alpha k_\lambda(t)$;
on the second half each factor of $h(s)$ is bounded by its value
at $t/2$, which is comparable to the displayed expression.
This includes the first few time levels.

The first recent contribution is absorbed into $\tau k_\lambda(t)$,
since
\[
 \frac{\tau^{1+\alpha}t^{-1-\alpha}H_\lambda(t)}
      {\tau k_\lambda(t)}
 \le C_\alpha\left[\left(\frac\tau t\right)^\alpha
                              +\lambda\tau^\alpha\right]
 \le C_{\alpha,K}.
\]
Here $t=t_n\ge\tau/K$ and $\lambda\tau^\alpha\le1$ were used.
The other recent contribution is at most
$C\tau^{2-\alpha}t^{2\alpha-2}/(1+\lambda t^\alpha)$.
Substitution in \eqref{eq:subdiffusion-error-convolution} proves
the finite-time theorem.

\end{proof}

\subsection{Long time estimates for forcing with a stationary limit}
\label{subsec:subdiffusion-infinite}

Suppose $f(t)\to f_\infty$ in $L^2(\Omega)$ as $t\to\infty$,
and set $u_\infty=A^{-1}f_\infty$. In addition to continuity at
zero and twice continuous differentiability for $t>0$, assume
\begin{equation}\label{eq:subdiffusion-regularity}
 \|\partial_t^k u(t)\|\le C_u
       \frac{t^{\alpha-k}}{(1+\lambda t^\alpha)^2},
 \qquad k=1,2,\quad t>0.
\end{equation}
Here $C_u$ is independent of the final observation time.
The transition from $t^{\alpha-k}$ to
$\lambda^{-2}t^{-\alpha-k}$ occurs at $t\asymp\lambda^{-1/\alpha}$.
Convergence of $f$ alone does not imply this derivative estimate:
its temporal decay and regularity must also be controlled.

\begin{proposition}[Sufficient data conditions]\label{prop:subdiffusion-regularity}
Let $A=-\Delta-\kappa I$ be the positive self-adjoint Dirichlet
realization in $L^2(\Omega)$, with $A\ge\lambda I$, $\lambda>0$.
If $f(t)\equiv f_\infty\in L^2(\Omega)$ and $u_0\in D(A)$,
then \eqref{eq:subdiffusion-regularity} holds with
\[
 C_u=C_\alpha\|Au_0-f_\infty\|.
\]
\end{proposition}

\begin{proof}
Set $u_\infty=A^{-1}f_\infty$ and $v_0=u_0-u_\infty\in D(A)$.
The stationary shift gives
$u(t)=u_\infty+E_\alpha(-t^\alpha A)v_0$.
For a spectral value $a\ge\lambda$, Lemma~\ref{d:lem:ml} gives
$r_a(t)\le C_\alpha a t^{\alpha-1}/(1+a t^\alpha)^2$.
Complete monotonicity also gives
\[
 -r_a'(t)\le\frac2t\bigl(r_a(t/2)-r_a(t)\bigr)
 \le\frac2t r_a(t/2)
 \le C_\alpha\frac{a t^{\alpha-2}}{(1+a t^\alpha)^2}.
\]
Indeed, $-r_a'$ is decreasing, so integrate it over $[t/2,t]$.
Consequently, for $k=1,2$,
\[
 |\partial_t^k E_\alpha(-a t^\alpha)|
 \le C_\alpha\frac{a t^{\alpha-k}}{(1+a t^\alpha)^2}
 \le C_\alpha\frac{a t^{\alpha-k}}{(1+\lambda t^\alpha)^2}.
\]
The spectral theorem and $Av_0=Au_0-f_\infty$ prove the claim.
\end{proof}

For time-dependent forcing, write $f=f_\infty+g$. It suffices
that $g$ is twice continuously differentiable in time with values
in $L^2(\Omega)$ and satisfies
\[
 \|g^{(j)}(t)\|\le G(1+t)^{-\alpha-j},
 \qquad j=0,1,2,\quad t\ge0.
\]
Together with $u_0\in D(A)$, this implies
\eqref{eq:subdiffusion-regularity} with
$C_u\le C_{\alpha,\lambda}(\|Au_0-f_\infty\|+G)$,
independently of the final time; no condition $Au_0=f(0)$ is
required. Indeed, the additional solution is $z=K_A*g$, where
$K_A(t)=A^{-1}r_A(t)$ and
$\|K_A^{(j)}(t)\|\le C_\alpha t^{\alpha-1-j}/(1+\lambda t^\alpha)^2$
for $j=0,1,2$. The identities
\[
 z'=K_A(t)g(0)+K_A*g',\qquad
 z''=K_A'(t)g(0)+K_A(t)g'(0)+K_A*g''
\]
give $\|z^{(k)}(t)\|\le C_\alpha Gt^{\alpha-k}$ for $0<t\le1$.
For $t\ge1$, split the convolutions at $t/2$ and use
$\int_0^\infty\|K_A(s)\|\,ds\le C_\alpha/\lambda$.
This directly bounds $z'$ by $C_{\alpha,\lambda}Gt^{-\alpha-1}$.
For $z''$, the cancellation
\[
 K_A(t)g'(0)+\int_0^{t/2} K_A(t-s)g''(s)\,ds
 =K_A(t)g'(h)
   +\int_0^{t/2}[K_A(t-s)-K_A(t)]g''(s)\,ds
\]
and $\int_0^\infty s\|g''(s)\|\,ds\le C_\alpha G$
bound the early part by $C_{\alpha,\lambda}Gt^{-\alpha-2}$;
the recent part has the same bound since
$\sup_{s\ge t/2}\|g''(s)\|\le C_\alpha Gt^{-\alpha-2}$.
Combining these estimates with the stationary-shift argument
proves the asserted all-time regularity.

We remark that the condition $u_0\in D(A)$ is a sufficient spatial regularity
assumption; $u_0\in L^2$ alone generally gives weaker small-time scale $t^{-k}$.

\begin{theorem}\label{thm:subdiffusion-decay}
Assume \eqref{eq:subdiffusion-regularity}, $\lambda>0$,
$\tau/\min_j\tau_j\le K$, and $\lambda\tau^\alpha\le1$.
Continue the L1 scheme to arbitrary positive times. Its consistency
error and temporal error satisfy, for every $t_n>0$,
\begin{equation}\label{eq:subdiffusion-consistency-split}
 \|\eta_n\|\le C\left[
       (\tau/t_n)^{1+\alpha}
       +\frac{(\tau/t_n)^{2-\alpha}}{(1+\lambda t_n^\alpha)^2}\right],
\end{equation}
\begin{equation}\label{eq:subdiffusion-decay}
\begin{aligned}
 \|e_n\|\le{}&\|e_0\|E_\alpha(-\rho^{-1}\lambda t_n^\alpha)\\
 &+C\left[
 \frac{\tau t_n^{\alpha-1}}{(1+\lambda t_n^\alpha)^2}
 +\frac{\tau^{2-\alpha}t_n^{2\alpha-2}}
              {(1+\lambda t_n^\alpha)^3}\right].
\end{aligned}
\end{equation}
Here $\rho$ is as in
\eqref{eq:discrete-gronwall} and $C$ is the product of the constant in \eqref{eq:discrete-gronwall} with $C_{u}$; in particular, the estimates are uniform in the final time.
\end{theorem}

\begin{proof}
Repeat the proof of Lemma~\ref{lem:subdiffusion-consistency-finite}.
The estimate on $[0,t_n/2]$ is unchanged. On $[t_n/2,t_n]$,
\eqref{eq:subdiffusion-regularity} inserts the additional factor
$(1+\lambda t_n^\alpha)^{-2}$ in the interpolation bound.
This proves \eqref{eq:subdiffusion-consistency-split}; the first
few levels are again controlled by its first term alone.
For the global error, use \eqref{eq:subdiffusion-error-convolution}
with the majorant
\[
 \tilde{h}(s)=(\tau/(s+\tau))^{1+\alpha}
      +\frac{(\tau/(s+\tau))^{2-\alpha}}
                   {(1+\lambda s^\alpha)^2}.
\]
As before, $\int_0^\infty \tilde{h} \le C_\alpha\tau$. Splitting at $t_n/2$
therefore bounds the consistency contribution by
\[
 C\left[\tau k_\lambda(t_n)
 +\left(\tau^{1+\alpha}t_n^{-1-\alpha}
 +\frac{\tau^{2-\alpha}t_n^{\alpha-2}}
             {(1+\lambda t_n^\alpha)^2}\right)H_\lambda(t_n)\right].
\]
The absorption in the finite-time proof applies unchanged to the
first recent term. The second gains one more denominator from
$H_\lambda$, proving \eqref{eq:subdiffusion-decay}.
\end{proof}

\subsection{Discussion}

The bounds in Sections~\ref{subsec:subdiffusion-finite}
and~\ref{subsec:subdiffusion-infinite} both separate the accumulated
initial singularity from recent interpolation errors.

For exact initial values, $e_0=0$. Their first-order scale is 
\[
B_1=\tau t_n^{\alpha-1}/(1+\lambda t_n^\alpha)^2.
\]
The respective higher-order scales and ratios are
\begin{equation}\label{eq:subdiffusion-regime-ratio}
\begin{aligned}
 B_{2,\mathrm{fin}}&=\frac{\tau^{2-\alpha}t_n^{2\alpha-2}}
                            {1+\lambda t_n^\alpha},
 &\frac{B_{2,\mathrm{fin}}}{B_1}
     &=(\tau/t_n)^{1-\alpha}(1+\lambda t_n^\alpha),\\
 B_{2,\infty}&=\frac{\tau^{2-\alpha}t_n^{2\alpha-2}}
                            {(1+\lambda t_n^\alpha)^3},
 &\frac{B_{2,\infty}}{B_1}
     &=\frac{(\tau/t_n)^{1-\alpha}}{1+\lambda t_n^\alpha}.
\end{aligned}
\end{equation}
Clearly, when $t_n$ is small or if the time-varying forcing term $f$ dominates the decay feature,
 one may observe the $O(\tau^{2-\alpha})$ convergence.  
However, refinement at
fixed $T$ and $\lambda$ favors the first-order scale. This is consistent with the parameter-dependent trends in
\cite{zhangzhangzhao-decay}, although an upper-bound comparison
does not establish an exact observed order.

 Under
\eqref{eq:subdiffusion-regularity}, recent errors instead acquire
the extra decay factor shown in $B_{2,\infty}$.
For $\lambda t_n^\alpha\ge1$, \eqref{eq:subdiffusion-decay} gives
\begin{equation}\label{eq:subdiffusion-longtime}
 \|e_n\|\le\|e_0\|E_\alpha(-\rho^{-1}\lambda t_n^\alpha)
 +C\left[
 \lambda^{-2}\tau t_n^{-1-\alpha}
 +\lambda^{-3}\tau^{2-\alpha}t_n^{-2-\alpha}\right].
\end{equation}
This indicates that if $t_n\gg 1$, the recent interpolation
errors are small due to the decay of the solution, so the accumulated initial error dominates, 
and one will observe the first order convergence.

We remark that constants $M_T$ and $C$ might also be controlled for more careful comparison about the parameters.
 The results provide the two-term L1 structure motivated by Conjecture~4.1
of \cite{zhangzhangzhao-decay}; they do not resolve the non-dissipative case $\kappa\ge\lambda_1$.

\subsection*{Acknowledgement}

The work of L.L. is partially supported by NSFC 12371400. ChatGPT Astra and 6.1 have helped in constructing the supersolution based on our old version of draft, and in writing first version of manuscript.

\appendix

\section{Proof of Lemma \ref{lem:kernelcont}}\label{app:phi}

\begin{proof}

Write
$\phi=\phi_{\lambda,\kappa}$ and $Q=Q_{\lambda,\kappa}$.
The solution $\phi$ is continuous on $[0,\infty)$ and smooth on $(0,\infty)$.
Integrating the equation from $0$ to $t$ gives
\begin{equation}\label{proof:phi:integralidentity}
    \kappa\phi(t)+I^{1-\alpha}\phi(t)
    +\lambda\int_0^t\phi(s)\,\dd s=\lambda.
\end{equation}
Take the principal branch of $z^\alpha$. Taking the Laplace transform and multiplying $z$ gives
$(z^\alpha+\kappa z+\lambda)\widehat\phi(z)=\lambda$, which proves the first assertion.

Using the formula,  $\widehat\phi$ is analytic in
$\mathbb C\setminus(-\infty,0]$.
A standard keyhole contour argument around the negative
real axis gives
\begin{equation}\label{eq:phiLapContour}
    \widehat\phi(z)
    =\int_0^\infty
      \frac{\rho_{\lambda,\kappa}(r)}{z+r}\,\dd r,
    \qquad z>0,
\end{equation}
where
\begin{equation}\label{eq:density}
    \rho_{\lambda,\kappa}(r)
    =\frac{\lambda\sin(\pi\alpha)}{\pi}
      \frac{r^\alpha}
      {(\lambda-\kappa r+r^\alpha\cos(\pi\alpha))^2
       +r^{2\alpha}\sin^2(\pi\alpha)}.
\end{equation}
The density is positive and integrable.
By uniqueness of the Laplace transform,
\begin{equation}\label{eq:phiLapconc2}
    \phi(t)
    =\int_0^\infty
      e^{-rt}\rho_{\lambda,\kappa}(r)\,\dd r.
\end{equation}
Differentiation under the integral sign yields
\begin{equation}\label{eq:phicomppos0}
    (-1)^p\phi^{(p)}(t)
    =\int_0^\infty
      e^{-rt}r^p\rho_{\lambda,\kappa}(r)\,\dd r>0,
    \qquad t>0,\quad p=0,1,2.
\end{equation}
Thus, $\phi$ is positive, decreasing, and convex.
Moreover, monotone convergence verifies that
\begin{equation}
    \phi(0+)=\lim_{z\to\infty}z\widehat\phi(z)
      =\frac{\lambda}{\kappa}, \qquad
    \int_0^\infty\phi(t)\,\dd t=\lim_{z\downarrow0}\widehat\phi(z)=1.
\end{equation}

We next establish \eqref{eq:phiasymp}.
For $t>0$ satisfying $\kappa t^{\alpha-1}\le1$, set
$\sigma:=\kappa t^{\alpha-1}$ and $\theta:=\lambda t^\alpha$.
Changing variables $y=rt$ in \eqref{eq:phicomppos0} gives
\begin{equation}
    (-1)^p\phi^{(p)}(t)
    =\frac{\sin(\pi\alpha)}{\pi}
      \lambda t^{\alpha-1-p}I_p(\sigma,\theta),
    \qquad p=0,1,2,
\end{equation}
where
\begin{equation}\label{eq:Ip}
    I_p(\sigma,\theta)
    :=\int_0^\infty
    \frac{e^{-y}y^{\alpha+p}}
    {(\theta-\sigma y+\cos(\pi\alpha)y^\alpha)^2
     +\sin^2(\pi\alpha)y^{2\alpha}}\,\dd y.
\end{equation}
Denote the denominator by $D(y)$.
Restricting the integral to $[1,2]$, where
$D(y)\le C(1+\theta)^2$, yields
\begin{equation}
    I_p(\sigma,\theta)\ge c_\alpha(1+\theta)^{-2}.
\end{equation}
For the upper bound, observe that
\begin{equation}
    D(y)\ge c_\alpha
    \bigl((\theta-\sigma y)^2+y^{2\alpha}\bigr).
\end{equation}
For $\theta\le1$, this directly gives $I_p\le C_\alpha$.
For $\theta\ge1$, split the integral at $y=\theta/2$.
Using $\theta-\sigma y\ge\theta/2$ on the first interval
and an exponential tail estimate on the second, we obtain
\begin{equation}
    I_p(\sigma,\theta)
    \le C_\alpha\left(
        \theta^{-2}\int_0^{\theta/2}
        e^{-y}y^{\alpha+p}\,\dd y
        +\int_{\theta/2}^\infty
        e^{-y}y^{p-\alpha}\,\dd y
    \right)
    \le C_\alpha\theta^{-2}.
\end{equation}
Consequently,
\begin{equation}
    I_p(\sigma,\theta)\asymp_\alpha(1+\theta)^{-2},
    \qquad p=0,1,2.
\end{equation}
Substituting this estimate into the representation of
$(-1)^p\phi^{(p)}(t)$ proves \eqref{eq:phiasymp}.

For the cumulative estimate, the spectral representation
and the elementary bound
$1-e^{-x}\asymp x/(1+x)$ for $x>0$ yield
\begin{equation}\label{eq:Q-laplace-comparison}
    \begin{aligned}
        Q(t)
        &=\int_0^\infty
          \frac{1-e^{-rt}}{r}
          \rho_{\lambda,\kappa}(r)\,\dd r\\
        &\asymp
          \int_0^\infty
          \frac{\rho_{\lambda,\kappa}(r)}{r+t^{-1}}\,\dd r
          =\widehat\phi(1/t)
          =\frac{\theta}{1+\sigma+\theta}
          \asymp\frac{\theta}{1+\theta}.
    \end{aligned}
\end{equation}
The last equivalence follows from $\sigma\le1$.
This proves \eqref{eq:Qasym}.
\end{proof}

\bibliographystyle{plain}
\bibliography{references}

@article{ramirez2012generalized,
  title={Generalized monotone iterative technique for {Caputo} fractional differential equation with periodic boundary condition via initial value problem},
  author={Ram{\'\i}rez, J Diego and Vatsala, Aghalaya S},
  journal={International Journal of Differential Equations},
  volume={2012},
  pages={1--17},
  year={2012},
  publisher={Hindawi Limited}
}

@article{fllx18note,
  title={A note on one-dimensional time fractional {ODEs}},
  author={Feng, Y. and Li, L. and Liu, J.-G. and Xu, X.},
  journal={Appl. Math. Lett. },
  volume={83},
  pages={87--94},
  year={2018}
}

@article{liliu2018compact,
  title={Some compactness criteria for weak solutions of time fractional {PDE}s},
  author={Li, L. and Liu, J.-G.},
  journal={SIAM J. Math. Anal.},
  volume={50},
  number={4},
  pages={3963--3995},
  year={2018},
  publisher={SIAM}
}

@article{liliu18frac,
  title={A generalized definition of {Caputo} derivatives and its application to fractional {ODEs}},
  author={Li, L. and Liu, J.-G.},
  journal={SIAM: J. Math. Anal.},
  volume={50},
  number={3},
  year={2018}
}

@article{liao2019discrete,
  title={A discrete {Gr\"onwall} inequality with applications to numerical schemes for subdiffusion problems},
  author={Liao, Hong-lin and McLean, William and Zhang, Jiwei},
  journal={SIAM Journal on Numerical Analysis},
  volume={57},
  number={1},
  pages={218--237},
  year={2019},
  publisher={SIAM}
}

@article{zhangzhangzhao-decay,
  author = {Zhang, Jiwei and Zhang, Zhimin and Zhao, Chengchao},
  title = {Do we need decay-preserving error estimate for solving
           parabolic equations with initial singularity?},
  journal = {Communications in Computational Physics},
  year    = {2026},
  volume  = {39},
  number  = {4},
  pages   = {1027--1058}
}

@article{paper,
  author  = {Feng, Yuanyuan and Li, Lei and Liu, Jian-Guo and Tang, Tao},
  title   = {Some {Gr\"onwall} Inequalities for a Class of
             Discretizations of Time Fractional Equations
             on Nonuniform Meshes},
  journal = {SIAM Journal on Numerical Analysis},
  year    = {2024},
  volume  = {62},
  number  = {5},
  pages   = {2196--2221},
  doi     = {10.1137/24M1631614}
}

@article{paper2,
  author  = {Li, Lei and Wang, Dongling},
  title   = {Numerical stability of {Gr\"unwald--Letnikov} method
             for time fractional delay differential equations},
  journal = {BIT Numerical Mathematics},
  year    = {2022},
  volume  = {62},
  number  = {3},
  pages   = {995--1027},
  doi     = {10.1007/s10543-021-00900-0}
}

@article{fengli2023a,
  title={On the completely positive kernels for nonuniform meshes},
  author={Feng, Yuanyuan and Li, Lei},
  journal={Quarterly of Applied Mathematics},
  year   = {2024},
  volume = {82},
  number = {4},
  pages  = {801--817},
  doi    = {10.1090/qam/1684}
}

@article{liao2020positive,
  author  = {Liao, Hong-lin and Tang, Tao and Zhou, Tao},
  title   = {Positive Definiteness of Real Quadratic Forms
             Resulting from the Variable-Step {L1}-Type
             Approximations of Convolution Operators},
  journal = {Science China Mathematics},
  year    = {2024},
  volume  = {67},
  number  = {2},
  pages   = {237--252},
  doi     = {10.1007/s11425-022-2229-5}
}

@article{feng2018continuous,
  title={Continuous and discrete one dimensional autonomous fractional {ODE}s.},
  author={Feng, Y.  and  Li, L. and Liu, J.-G. and Xu, X.},
  journal={Discrete Contin. Dyn. Syst. Ser. B},
  volume={23},
  number={8},
  year={2018}
}

@article{clement1981asymptotic,
  title={Asymptotic behavior of solutions of nonlinear {Volterra} equations with completely positive kernels},
  author={Cl{\'e}ment, Ph. and Nohel, John A.},
  journal={SIAM Journal on Mathematical Analysis},
  volume={12},
  number={4},
  pages={514--535},
  year={1981},
  publisher={SIAM}
}

@article{miller1968volterra,
  title={On {Volterra} integral equations with nonnegative integrable resolvents},
  author={Miller, R. K.},
  journal={Journal of Mathematical Analysis and Applications},
  volume={22},
  number={2},
  pages={319--340},
  year={1968},
  publisher={Academic Press}
}

@book{widder41,
	Author = {Widder, David Vernon},
	Publisher = {Princeton University Press},
	Title = {The Laplace Transform},
	Year = {1941}}

@book{ssv12,
  title={Bernstein functions: theory and applications},
  author={Schilling, R. L. and Song, R. and Vondracek, Z.},
  volume={37},
  year={2012},
  publisher={Walter de Gruyter}
}

@article{gorenflo2002,
  author  = {Gorenflo, Rudolf and Loutchko, Joulia and Luchko, Yuri},
  title   = {Computation of the {Mittag--Leffler} Function
             {$E_{\alpha,\beta}(z)$} and Its Derivative},
  journal = {Fractional Calculus and Applied Analysis},
  year    = {2002},
  volume  = {5},
  number  = {4},
  pages   = {491--518}
}

@article{liwang2021cm,
  author = {Li, Lei and Wang, Dongling},
  title = {Complete monotonicity-preserving numerical methods for time fractional {ODEs}},
  journal = {Communications in Mathematical Sciences},
  year = {2021},
  volume = {19},
  number = {5},
  pages = {1301--1336},
  doi = {10.4310/CMS.2021.v19.n5.a6}
}

@article{liliu2019caputo,
  author = {Li, Lei and Liu, Jian-Guo},
  title = {A Discretization of {Caputo} Derivatives with Application to Time Fractional {SDEs} and Gradient Flows},
  journal = {SIAM Journal on Numerical Analysis},
  year = {2019},
  volume = {57},
  number = {5},
  pages = {2095--2120},
  doi = {10.1137/19M123854X}
}

@article{liao2018sharp,
  author  = {Liao, Hong-lin and Li, Dongfang and Zhang, Jiwei},
  title   = {Sharp Error Estimate of the Nonuniform {L1} Formula
             for Linear Reaction-Subdiffusion Equations},
  journal = {SIAM Journal on Numerical Analysis},
  year    = {2018},
  volume  = {56},
  number  = {2},
  pages   = {1112--1133},
  doi     = {10.1137/17M1131829}
}

@article{wangzou2023stability,
  author  = {Wang, Dongling and Zou, Jun},
  title   = {{Mittag--Leffler} Stability of Numerical Solutions
             to Time Fractional {ODEs}},
  journal = {Numerical Algorithms},
  year    = {2023},
  volume  = {92},
  number  = {4},
  pages   = {2125--2159},
  doi     = {10.1007/s11075-022-01379-8}
}

@article{stynes2017error,
  author  = {Stynes, Martin and O'Riordan, Eugene
             and Gracia, Jos{\'e} Luis},
  title   = {Error Analysis of a Finite Difference Method
             on Graded Meshes for a Time-Fractional Diffusion Equation},
  journal = {SIAM Journal on Numerical Analysis},
  year    = {2017},
  volume  = {55},
  number  = {2},
  pages   = {1057--1079},
  doi     = {10.1137/16M1082329}
}

@article{fengli2024monotonicity,
  author  = {Feng, Yuanyuan and Li, Lei},
  title   = {A Class of Monotonicity-Preserving Variable-Step
             Discretizations for {Volterra} Integral Equations},
  journal = {BIT Numerical Mathematics},
  year    = {2024},
  volume  = {64},
  note    = {Article 24},
  doi     = {10.1007/s10543-024-01027-8}
}

@book{diethelm2010analysis,
  author    = {Diethelm, Kai},
  title     = {The Analysis of Fractional Differential Equations:
               An Application-Oriented Exposition Using Differential
               Operators of {Caputo} Type},
  series    = {Lecture Notes in Mathematics},
  volume    = {2004},
  publisher = {Springer},
  address   = {Berlin, Heidelberg},
  year      = {2010},
  doi       = {10.1007/978-3-642-14574-2}
}

@book{jin2021fractional,
  author    = {Jin, Bangti},
  title     = {Fractional Differential Equations:
               An Approach via Fractional Derivatives},
  series    = {Applied Mathematical Sciences},
  volume    = {206},
  publisher = {Springer},
  address   = {Cham},
  year      = {2021},
  doi       = {10.1007/978-3-030-76043-4}
}

@article{vergara2015decay,
  author  = {Vergara, Vicente and Zacher, Rico},
  title   = {Optimal Decay Estimates for Time-Fractional and Other
             Nonlocal Subdiffusion Equations via Energy Methods},
  journal = {SIAM Journal on Mathematical Analysis},
  year    = {2015},
  volume  = {47},
  number  = {1},
  pages   = {210--239},
  doi     = {10.1137/130941900}
}

@article{kopteva2019error,
  author  = {Kopteva, Natalia},
  title   = {Error Analysis of the {L1} Method on Graded and Uniform
             Meshes for a Fractional-Derivative Problem in Two
             and Three Dimensions},
  journal = {Mathematics of Computation},
  year    = {2019},
  volume  = {88},
  number  = {319},
  pages   = {2135--2155},
  doi     = {10.1090/mcom/3410}
}

@article{lubich1986discretized,
  author  = {Lubich, Christian},
  title   = {Discretized Fractional Calculus},
  journal = {SIAM Journal on Mathematical Analysis},
  year    = {1986},
  volume  = {17},
  number  = {3},
  pages   = {704--719},
  doi     = {10.1137/0517050}
}

@article{wangxiaozou2020long,
  author  = {Wang, Dongling and Xiao, Aiguo and Zou, Jun},
  title   = {Long-Time Behavior of Numerical Solutions
             to Nonlinear Fractional {ODEs}},
  journal = {ESAIM: Mathematical Modelling and Numerical Analysis},
  year    = {2020},
  volume  = {54},
  number  = {1},
  pages   = {335--358},
  doi     = {10.1051/m2an/2019055}
}

\end{document}